\documentclass{article}
\usepackage[utf8]{inputenc}

\usepackage{amsmath,amssymb,amsthm,bbm}
\usepackage{mathtools}
\usepackage{amsfonts}
\usepackage{enumerate}
\usepackage{xcolor}
\usepackage[margin=1in]{geometry}

\newcommand\nc\newcommand
\nc{\bs}{\boldsymbol}
\nc{\com}[1]{\textcolor{red}{[#1]}}

\nc{\PP}{\mathbb{P}}
\newcommand{\R}{\mathbb R}
\nc{\ind}{\mathbbm{1}}
\newcommand{\E}{\mathbb E}

\newcommand{\Cov}{\mathrm{Cov}}

\newcommand{\Lip}{\mathrm{Lip}}

\newcommand{\bary}{\mathrm{bar}}
\newcommand{\Tr}{\mathrm{Tr}}
\newcommand{\e}{\mathrm e}

\newcommand{\comment}[1]{}
\DeclareMathOperator{\Ent}{Ent}
\DeclareMathOperator{\Var}{Var}
\DeclarePairedDelimiter{\abs}{\lvert}{\rvert}
\DeclarePairedDelimiter\norm{\lVert}{\rVert}

\newtheorem{thm}{Theorem}
\newtheorem{lem}[thm]{Lemma}
\newtheorem{prop}[thm]{Proposition}
\newtheorem{cor}[thm]{Corollary}
\newtheorem{conj}[thm]{Conjecture}

\theoremstyle{remark}

\theoremstyle{definition}
\newtheorem{defi}[thm]{Definition}

\newenvironment{rmk}{\textit{Remark :}}{}

\title{On the Log-Sobolev Constant of Log-Concave Vectors}
\author{Pierre Bizeul}
\date{}

\begin{document}

\maketitle
\begin{abstract}
It is well known that if a random vector satisfies a log-Sobolev inequality, all of its marginals have subgaussian tails. In the spirit of the KLS conjecture, we investigate whether this implication can be reversed under a log-concavity assumption. In the general setting, we improve on a result of Bobkov, establishing the best dimension dependent bound on the log-Sobolev constant of subgaussian log-concave measures, and we investigate some special cases.
\end{abstract}

\section{Introduction and results}

A Borel probability measure $\mu$ on $\R^n$ is said to satisfy a logarithmic Sobolev inequality with constant $\rho>0$ if for any smooth function $f:\R^n\to\R$, one has
\begin{equation}\label{eq:logsob_def}
\Ent_\mu(f^2) \leq 2\rho^2 \int_{\R^n} \abs{\nabla f}^2\, d\mu,
\end{equation}
where for a nonnegative function $g$,
\[
\Ent_\mu(g) = \E_\mu[g\log g] - \E_\mu[g]\log \E_\mu[g],
\]
and $\E_\mu(h) = \int_{\R^n} h\, d\mu$ denotes expectation with respect to $\mu$. 
Here $\abs{\cdot}$ denotes the Euclidean norm.
We denote by $\rho_{LS}(\mu)$ the optimal constant $\rho$ such that \eqref{eq:logsob_def} holds.
It is well known that the log-Sobolev inequality implies Gaussian concentration: the Herbst argument yields a quadratic bound on the logarithmic Laplace transform of Lipschitz functions.
\begin{equation}\label{eq93}
    \log \int e^{s f} \, d\mu 
    \ \leq\ 
    \frac{\rho^2 s^2 \, \vert f \vert_{\mathrm{Lip}}^2}{2}
    \ +\ s \int_{\mathbb{R}^n} f \, d\mu,
    \qquad s \in \mathbb{R},
\end{equation}
where $\vert f \vert_{\Lip}$ is the Lipschitz constant of $f$. Markov's inequality then implies Gaussian concentration of $f$ around its mean,
\begin{equation}{\label{eq:gaussian_conc}}
    \mu(\vert f - \E_\mu(f)\vert\geq t) \leq 2e^{-\frac{t^2}{2 \vert f \vert_{\Lip}^2\rho^2}}.
\end{equation}
Taking $f$ to be a linear form, we deduce that every one-dimensional marginal of $\mu$ is subgaussian. Classically, the log-Laplace estimate~\eqref{eq93} can be reformulated in a way that will be convenient for our purposes. Namely, if $f$ is a $1$-Lipschitz function satisfying $\int f\, d\mu = 0$, then 
\begin{equation}\label{eq101}
    \E \e^{ \frac{f^2}{4\rho_{LS}^2(\mu)}} \leq 2.
\end{equation}
Since we will use to this fact again later in the introduction, we recall its brief proof below.
Let $X$ be a random variable satisfying a log-Laplace bound of the form $\E e^{sX} \leq e^{s^2/2}$ for all $s\in\R$. Let $Z$ be a standard Gaussian, independent from $X$. Recall the elementary identities $\E e^{\lambda Z} = e^{\lambda^2/2}$ for $\lambda\in\R$ and $\E e^{t Z^2} = (1-2t)^{-1/2}$, for $t<1$. By Fubini,
$$\E e^{X^2/4} = \E e^{X Z/\sqrt{2}} \leq \E e^{Z^2/4} = \sqrt{2} < 2,$$
which is what we wanted.

In a related direction, we say that $\mu$ satisfy a Poincaré inequality with constant $K>0$ if for any smooth function $f: \R^n\mapsto\R$, one has
\begin{equation}{\label{eq:Poincaré_def}}
\Var_\mu(f) \leq K^2 \ \int_{\R^n} \abs{\nabla f}^2 d\mu,
\end{equation}
where $Var_\mu(f) = \E_\mu(f^2) - \E_\mu(f)^2$ is the variance of $f$. We denote by $C_{P}(\mu)$ the optimal constant $K$ such that \eqref{eq:Poincaré_def} holds. It is classical that the log-Sobolev inequality \eqref{eq:logsob_def} is stronger than \eqref{eq:Poincaré_def} : $C_{P}(\mu) \leq \rho_{LS}(\mu)$.
Furthermore, while log-Sobolev implies Gaussian concentration, it was first observed by Gromov and Milman \cite{gromov1983topological} that Poincaré inequalities implies exponential concentration. Namely, using e.g \cite[Proposition~4.1]{bobkov1997poincare} one can deduce that for any $1$-Lipschitz function with $\int fd\mu=0$ 
\begin{equation}\label{eq109}
    \E_\mu e^{\frac{\abs{f}}{3C_P(\mu)}} \leq 2.
\end{equation}
By Markov's inequality, this implies that Lipschitz functions are exponentially concentrated around their mean :
$$\mu(\vert f - \E_\mu(f)\vert\geq t) \leq 2e^{-\frac{t}{3C_P(\mu)}}.$$
It will be convenient to express Gaussian and exponential concentration in terms of Orlicz norms. Let $\psi:\mathbb{R}_+\to\mathbb{R}_+$ be a Young function, that is a nonconstant convex increasing function with $\psi(0)=0$. For a real random variable $Y$, the corresponding Orlicz norm is defined by
\[
\|Y\|_{\psi} = \inf\Bigl\{\, t>0 : \E\,\psi\!\left(\frac{|Y|}{t}\right) \le 1 \Bigr\}.
\]
Let now $\mu$ be a probability measure on $\R^n$, and let $g:\R^n\to\R$ be a measurable function. If $X\sim\mu$, we set
\[
\|g\|_{L^{\psi}(\mu)} := \|g(X)\|_{\psi}.
\]
We will recurrently use the functions $\psi_1(t) = e^{t}-1$ and $\psi_2(t) = e^{t^2}-1$. The bound \eqref{eq109} may be rephrased as 
\begin{equation}\label{eq:poinc_psi}
\norm{f-E_\mu f}_{\psi_1}\leq 3C_P(\mu),
\end{equation}for any $1$-Lipschitz function $f$. On the other hand if $\mu$ satisfies a log-Sobolev inequality, the bound \eqref{eq101} is equivalent to
\begin{equation}\label{eq:log_psi2}
\norm{f-E_\mu f}_{\psi_2}\leq 2\rho_{LS}(\mu).
\end{equation}

Not all measures may satisfy a Poincaré or log-Sobolev inequality. Even under good integrability assumptions, if the support of $\mu$ is disconnected, one may build a non-constant function whose gradient vanishes $\mu$ almost everywhere, violating \eqref{eq:Poincaré_def}, hence also \eqref{eq:logsob_def}. A general class of measures which avoid double-bump type distribution is the class of log-concave measures, that is, measures that write $d\mu = e^{-V(x)}dx$, for some convex $V:\R^n\mapsto\R\cup\{+\infty\}$. For such measures, the well-known KLS conjecture proposes that it is enough, up to a universal constant, to test linear functions in \eqref{eq:Poincaré_def}.

\begin{conj}[KLS]\label{conj:KLS}
There exists a constant $C>0$ such that for any $n\in \mathbb{N}$ and any log-concave probability measure $\mu$ on $\R^n$,
\begin{equation}\label{eq_104}
    C_P^2(\mu) \leq C \sup_{\theta \in \mathbb{S}^{n-1}}\Var_\mu\left(\langle\ .\ ,\ \theta\ \rangle\right). 
\end{equation}
\end{conj}
We will write $\Gamma_n$ for the best (possibly) dimensional constant in \eqref{eq_104}. The KLS conjecture has attracted a lot of attention since its original formulation in \cite{kannan1995isoperimetric}, culminating in a logarithmic estimate by Klartag \cite{klartag2023logarithmic}. By analogy, it is natural to conjecture that the log-Sobolev constant of log-concave probabilities should be controlled by the $L^{\psi_2}$ norm of linear functionals. 
\begin{conj}\label{conj:KLSob}
There exists a universal constant $C>0$ such that for every 
$n \in \mathbb{N}$ and every centered log-concave probability measure 
$\mu$ on $\mathbb{R}^n$,
\[
\rho_{LS}(\mu) 
\leq C \sup_{\theta \in \mathbb{S}^{n-1}} 
\| \langle \cdot , \theta \rangle \|_{L^{\psi_2}(\mu)}.
\]
\end{conj}

It is convenient to introduce the following terminology.

\begin{defi}
Let $\mu$ be a probability measure on $\mathbb{R}^n$, and let $b_\mu$ 
denote its barycenter.  
We say that $\mu$ is \emph{$\alpha$-subgaussian} if
\[
\sup_{\theta \in \mathbb{S}^{n-1}}
\| \langle \cdot - b_\mu , \theta \rangle \|_{L^{\psi_2}(\mu)} 
\leq \alpha.
\]
\end{defi}

For $n \ge 1$, define
\begin{equation}\label{eq:def_constant_klogsob}
G_n := \sup_{\mu} \rho_{LS}(\mu),
\end{equation}
where the supremum is taken over all log-concave $1$-subgaussian 
probability measures on $\mathbb{R}^n$.  
By scaling invariance of the log-Sobolev constant, 
Conjecture \ref{conj:KLSob} is equivalent to the assertion 
that $(G_n)_{n\geq1}$ is uniformly bounded.

\vspace{0.5em}

A theorem of Bobkov \cite{bobkov1999isoperimetric} states that 
if $\mu$ is a centered log-concave measure on $\mathbb{R}^n$, then
\[
\rho_{LS}(\mu) \lesssim\bigl\lVert\,|X|\,\bigr\rVert_{\psi_2},
\]
where $X$ is a random vector with law $\mu$.
As shown in Proposition \ref{prop:norm_subgauss} below, one always has
\[
\bigl\lVert\,|X|\,\bigr\rVert_{\psi_2} 
\;\lesssim\;
\sqrt{n}\, 
\sup_{\theta \in \mathbb{S}^{n-1}} \|\langle X,\theta\rangle\|_{\psi_2}.
\]
This estimate is sharp in general, since by Jensen’s inequality $
\bigl\lVert\,|X|\,\bigr\rVert_{\psi_2}^2 \; \ge \; \E|X|^2 .$
Consequently, Bobkov’s result implies $G_n \lesssim \sqrt{n}$.
Our first result is the following improvement:
\begin{thm}\label{thm:MAINLOGSOB} There exists a constant $C>0$ such that for all $n\geq1$
    $$G_n \leq C\ n^{1/4}.$$
\end{thm}

In this paper, we present two distinct routes to Theorem~\ref{thm:MAINLOGSOB}. 
More precisely, we establish two non-comparable estimates on the log-Sobolev constant of log-concave random vectors, 
Theorem~\ref{thm189} and Theorem~\ref{thm:LOCALISATIONLOGSOB} below. 
Theorem~\ref{thm:MAINLOGSOB} follows directly from Theorem~\ref{thm189}, 
while Theorem~\ref{thm:LOCALISATIONLOGSOB} provides an alternative proof up to a logarithmic factor.

\begin{thm}\label{thm189}
Let $\mu$ be a $1$-subgaussian log-concave probability measure on $\R^n$, and let $\Cov(\mu) = \E_\mu(xx^T) - \E_\mu(x)\E_\mu(x)^T$ denote its covariance matrix. Then
\[
\rho_{LS}(\mu) \lesssim 1 + \bigl(n\,\|\Cov(\mu)\|_{op}\bigr)^{1/4}.
\]
\end{thm}

For any $1$-subgaussian probability measure $\mu$ with barycenter $b_\mu = \E_\mu(x)$,
\[
\|\Cov(\mu)\|_{op}
= \sup_{\theta \in \mathbb{S}^{n-1}} \Var_\mu(\langle \cdot , \theta \rangle)
\le \sup_{\theta \in \mathbb{S}^{n-1}} \|\langle \cdot - b_\mu , \theta \rangle\|_{L^{\psi_2}(\mu)}^2
\le 1.
\]
Hence Theorem \ref{thm189} implies Theorem \ref{thm:MAINLOGSOB}. However, in many situations the operator norm of the covariance is much smaller, which leads to a sharper estimate. For example, if $\nu_1$ is the uniform measure on
\[
\sqrt{n}B_1^n = \Bigl\{x \in \R^n : \sum_{i=1}^n |x_i| \le \sqrt{n}\Bigr\},
\]
then an elementary computation shows that
\[
\sup_{\theta \in \mathbb{S}^{n-1}} \|\langle \cdot , \theta \rangle\|_{L^{\psi_2}(\nu_1)} \simeq 1,
\qquad
\|\Cov(\nu_1)\|_{op} \simeq \frac{1}{n}.
\]
Thus Theorem \ref{thm189} yields the sharp bound
\[
\rho_{LS}(\nu_1)
\simeq \sup_{\theta \in \mathbb{S}^{n-1}} \|\langle \cdot , \theta \rangle\|_{L^{\psi_2}(\nu_1)}
\simeq 1.
\]
In other words, our estimate for the log-Sobolev constant \emph{improves} precisely when there is a gap between the $L^2$ and $L^{\psi_2}$ norms of linear functionals. The proof of Theorem \ref{thm189} uses a general result about the concentration function of log-concave probabilities established in \cite{bizeul2022measures} via stochastic localization combined with an elementary deviation inequality for the euclidean norm of $1$-subgaussian vectors. The proof is given at the end of Section \ref{sec:background}.
\medskip

\comment{The proof of Theorem \ref{thm:MAINLOGSOB} uses a general result about the concentration function of log-concave probabilities established in \cite{bizeul2022measures} via stochastic localization. This latter process proved very successful in proving Poincaré inequalities, culminating in a logarithmic dimensional estimate for the Poincaré constant of an isotropic log-concave probability measure \cite{klartag2023logarithmic}. The study of log-Sobolev inequalities via stochastic localization was initiated in \cite{lee2017eldan}. We discuss the new difficulties that arise in this setting in Section \ref{sec:loc_sto}.}

We now turn to the second result from which Theorem \ref{thm:MAINLOGSOB} follows, up to a logarithmic factor. A classical result of Bobkov \cite{bobkov2007isoperimetric} states that for any log-concave random vector $X$,
\begin{equation}\label{eq:bobkov_spectral}
    C_P^2(X) \lesssim \Var(|X|^2)^{1/2}.
\end{equation}
This bound is tight (up to a universal constant) for some measures, such as the uniform distribution on some Euclidean ball, but is generally suboptimal; for instance, it yields $C_P(G) \lesssim n^{1/4}$ for a standard Gaussian vector $G$.
We prove the following analogue for the log-Sobolev constant.
\begin{thm}\label{thm:LOCALISATIONLOGSOB}
Let $X$ be a log-concave random vector. Then
\[
    \rho_{LS}^2(X) \lesssim \bigl\lVert\,|X|^2- \E\abs{X}^2\,\bigr\rVert_{\psi_1}.
\]
\end{thm}
As in Bobkov's inequality, the estimate is sharp up to constants when $X$ is uniform on the Euclidean ball. Indeed, if $X$ is uniformly distributed on the unit Euclidean ball, its norm is distributed as $F_n(t) = \PP(\abs{X}^2 \leq t) = t^{n/2}$. We deduce that $\E\abs{X}^2 = \int_0^1 \PP(\abs{X}^2 \geq t)\,dt = \frac{n}{n+2}$. Thus for any $t \geq \frac{1}{n+2}$,
\[
\PP\left(\big|\abs{X}^2 - \E\abs{X}^2\big| \geq t\right)
= \PP\left(\abs{X}^2 \leq \E\abs{X}^2 - t\right)
\leq \PP\left(\abs{X}^2 \leq 1 - t\right)
= (1 - t)^{n/2}
\leq e^{-nt/2}.
\]
Using the above bound, it is elementary to check that $\bigl\lVert\,|X|^2- \E\abs{X}^2\,\bigr\rVert_{\psi_1} \leq \frac{8}{n}$. Thus Theorem \ref{thm:LOCALISATIONLOGSOB} gives the correct estimate $\rho^2_{LS}(X) \simeq \frac{1}{n}$. On the other hand, for the standard Gaussian, Theorem \ref{thm:LOCALISATIONLOGSOB} only yields the suboptimal $\rho_{LS}(G) \lesssim n^{1/4}$.

In the general setting, using the identity
\[
\abs{X}^2 - \E\abs{X}^2
= \left(\abs{X} - \E\abs{X}\right)^2
+ 2\,\E\abs{X}\,\left(\abs{X} - \E\abs{X}\right)
- \Var(\abs{X})
\]
and the triangle inequality, we arrive at
\begin{equation}\label{eq243}
    \rho_{LS}^2(X) \;\lesssim\; \bigl\lVert\,|X|^2- \E\abs{X}^2\,\bigr\rVert_{\psi_1}
    \;\lesssim\;
    \bigl\|\,|X| - \E|X|\,\bigr\|_{\psi_2}^2
    \;+\;
    \E|X| \,\cdot\,
    \bigl\|\,|X| - \E|X|\,\bigr\|_{\psi_1},
\end{equation}
where we used that 
$\bigl\|\,Y^2\,\bigr\|_{\psi_1} = \bigl\|\,Y\,\bigr\|_{\psi_2}^2$ 
and the inequality 
$\Var(Y) \leq \bigl\|\,Y - \E Y\,\bigr\|_{\psi_2}^2$.
Let us discuss the two terms appearing on the right-hand side of~\eqref{eq243}. 
Assume that $X$ is a centered, $1$-subgaussian, log-concave random vector in $\R^n$. 
Since the $L^2$ norm is dominated by the $L^{\psi_2}$ norm (see Lemma~\ref{lem:psi_2_variance} below), 
the covariance matrix of $X$ is dominated by the identity, and therefore
\[
    \E|X|\;\leq\;\sqrt{n}.
\]
Moreover, the Euclidean norm is $1$-Lipschitz, so that
\[
    \bigl\|\,|X| - \E|X|\,\bigr\|_{\psi_1}
    \;\lesssim\;
    C_P(X)
    \;\lesssim\;
    \sqrt{\log n},
\]
where the last bound follows from the best known estimate on the KLS constant 
due to Klartag~\cite{klartag2023logarithmic}. 
Together, these observations allow us to control the second term in~\eqref{eq243} by $\sqrt{n}$, 
up to a logarithmic factor.

Finally, after a suitable preprocessing step, the first term can also be bounded by $\sqrt{n}$, 
using a deviation inequality of Gu\'edon and Milman~\cite{guedon2011interpolating}. 
The full argument is carried out in at the end of Section~\ref{sec:proof_localization}, where we establish:
\begin{cor}\label{cor_secondproof}
    Let $X$ be a $1$-subgaussian log-concave vector in $\R^n$, then $\rho_{LS}^2(X)\;\lesssim\;\sqrt{n\log n}.$
\end{cor}
In other words, we recover our main result, Theorem~\ref{thm:MAINLOGSOB}, 
up to a factor of $\log(n)^{1/4}$, which could be removed if the KLS conjecture were true.

\smallskip
Theorem~\ref{thm:LOCALISATIONLOGSOB} is proved in Section~\ref{sec:proof_localization} using a localization argument that reduces the problem to dimension~$1$. In the same section, we also provide a detailed proof of Corollary~\ref{cor_secondproof}.
\medskip

Next, we turn to some special cases. The first one is the class of rotationally invariant log-concave probabilities. For this class we show that Conjecture~\ref{conj:KLSob} holds. Using a result of Bobkov~\cite{bobkov2010gaussian}, we obtain:
\begin{thm}\label{thm:logsob_rot_inv}
Let $\mu$ be a rotationally invariant log-concave probability measure. Then there exists a universal constant $C>0$ such that
$$\rho_{LS}(\mu)\leq C\,\norm{\langle\cdot,e_1\rangle}_{L^{\psi_2}(\mu)},$$
where $e_1$ denotes the first element of the canonical basis.
\end{thm}
We now introduce two subclasses of subgaussian probability measures. Let $\mu$ be a probability measure on $\R^n$,  
for $t\ge 0$ and $h\in\R^n$, define
\[
L_t(h)
    = \log \int_{\R^n} \exp\!\left(h\cdot x - \tfrac{t}{2}|x|^2\right)\, d\mu(x),
\]
and the corresponding tilted measure
\[
\frac{d\mu_{t,h}}{d\mu}(x)
    = \exp\!\left(h\cdot x - \tfrac{t}{2}|x|^2 - L_t(h)\right).
\]
When $t=0$, the function $L_0$ is the log-Laplace transform of $\mu$. When it is convenient, we may use the notation $\tau_h$ for the operator linearly tilting measures by a vector $h$ : 
\begin{equation}\label{eq_notationtau}
    \tau_h \mu = \mu_{0,h}.
\end{equation}
A direct differentiation under the integral shows that the gradient and Hessian of $L_t$ are given by the barycenter and covariance matrix of the tilted measure $\mu_{t,h}$. For all $t\ge 0$ and $h\in\R^n$,
\[
\nabla_h L_t(h) = \mathrm{bar}(\mu_{t,h})
\qquad
\nabla_h^2 L_t(h) = \Cov(\mu_{t,h}).
\]
For symmetric matrices $A,B$, we write $A \preceq B$ if $B-A$ is positive semidefinite, and we denote by $I_n$ the $n\times n$ identity matrix.
\begin{defi}
We say that a probability measure $\mu$ on $\R^n$ is \emph{$\beta$ tilt-stable}, for some $\beta>0$, if for all $h\in\R^n$,
\begin{equation}\label{eq306}
    \nabla^2L_0(h) = \Cov(\mu_{0,h}) \preceq \beta^2I_n.
\end{equation}
We say that $\mu$ is \emph{$\beta$ strongly tilt-stable} if for all $h\in\R^n$ and all $t\geq0$,
\begin{equation}\label{eq310}
\nabla^2L_t(h) = \Cov(\mu_{t,h}) \preceq \beta^2I_n.
\end{equation}
\end{defi}
Using the same argument as in the proof of \eqref{eq101} one can verify that $1$-tilt stable probabilities are $2$-subgaussian.
\begin{lem}\label{lem315}
    Let $\mu$ be a $\beta$ tilt-stable probability measure, then $\mu$ is $2\beta$-subgaussian.    
\end{lem}
\begin{proof}
By scaling and translation invariance, we may assume that $\mu$ is $1$-tilt stable and centered. Integrating \eqref{eq306} we find that 
\begin{equation}\label{eq320}
    L_0(h)\leq \frac{\abs{h}^2}{2}
\end{equation}
Let $X$ be distributed according to $\mu$ and let $\theta$ be a unit vector. Denote $X_\theta = X\cdot\theta$, we want to check that $\E e^{X_\theta^2/4}\leq 2$. From \eqref{eq320} we know that $\E e^{sX_\theta} \leq e^{s^2/2}$, thus the situation is identical to the proof of \eqref{eq101} above:
$$\E e^{X_\theta^2/4} = \E e^{X_\theta Z/\sqrt{2}} \leq \E e^{Z^2/4} = \sqrt{2} < 2.$$
where $Z$ is and independent standard Gaussian.
\end{proof}
Finally, a related notion is the class of measures that are log-concave with respect to the standard Gaussian. We call them $1$-strongly log-concave. It is well known that such measures satisfy a Poincaré with constant $1$. In particular,
$$\norm{\Cov(\mu)}_{op} = \sup_{\theta \in \mathbb S^{n-1}}\Var_\mu(\langle .,\theta\rangle) \leq 1.$$
Observe that if $\mu$ is $1$-strongly log-concave, then so is $\mu_{t,h}$, for all $t\geq0$ and $h\in\R^n$. Thus the operator norm of the covariance matrix of $\mu_{t,h}$ is also bounded by $1$. In other words, $\mu$ is strongly tilt-stable. 
To summarize, we have a chain of inclusion
\begin{equation}\label{eq:hierarchy}
\bigl\{\, 1\text{-strongly log-concave} \,\bigr\}
\;\subset\;
\bigl\{\, 1\text{-strongly tilt-stable} \,\bigr\}
\;\subset\;
\bigl\{\, 1\text{-tilt-stable} \,\bigr\}
\;\subset\;
\bigl\{\, 2\text{-subgaussian} \,\bigr\}.
\end{equation}
The Bakry–Émery criterion~\cite{bakry2006diffusions} implies that if $\mu$ is $1$-strongly log-concave, then $\rho_{LS}(\mu) \leq 1$; see \cite{bobkov2000brunn} for a quick proof. In Section~\ref{sec:loc_sto}, we show that a similar bound holds under the weaker assumption of strong tilt-stability, a fact that is well known to experts.

\begin{lem}\label{lem:intro_strong_tilt_implies_logsob}
If $\mu$ is $\beta$-strongly tilt-stable and log-concave, then 
$\rho_{LS}(\mu) \leq 2\beta$.
\end{lem}

The rightmost and leftmost inclusions of \eqref{eq:hierarchy} are clearly not reversible. Indeed, the uniform measure on the cube is easily seen to be strongly tilt-stable (since every tilted measure $\mu_{t,h}$ remains a product measure), but it is of course not strongly log-concave. Moreover, the property of being tilt-stable is, in a sense, discontinuous. For $\alpha \in \R$, consider the measure $\nu_\alpha = \tau_\alpha \nu_0$ on $\R$, whose density is proportional to $d\nu_\alpha(x) = 1_{[-\sqrt{3},\,\sqrt{3}]}(x)\, e^{\alpha x}\,dx$. Recall that the tilt operator $\tau_\alpha$ was introduced in \eqref{eq_notationtau}. As $\abs{\alpha}\to\infty$, the measure $\nu_\alpha$ converges to a Dirac mass. One easily checks that the subgaussian constant satisfies $\norm{X_\alpha}_{\psi_2} \simeq (1+\abs{\alpha})^{-1}$ for $X_\alpha\sim\nu_\alpha$. However for every $\alpha\in\R$ we have $\Cov(\tau_{-\alpha}\nu_\alpha) = \Cov(\nu_0) = 1$. Hence each $\nu_\alpha$ is merely $1$–tilt-stable.

On the other hand, we prove in Section \ref{sec:subclass} that 1-tilt-stable log-concave probabilities are in fact strongly-tilt stable, with a dimension dependent constant. Let us introduce the quantities
\[
\Tilde{G_n} = \sup_{\mu}\rho_{LS}(\mu),
\qquad
    \Tilde{K}_n
    \;=\;
    \sup_{\mu}
    \bigl\|\,|X| - \E|X|\,\bigr\|_{\psi_2},
\]
where both suprema run over all log-concave probabilities $\mu$ on $\R^n$ that are $1$ tilt-stable. We prove the following
\begin{lem}\label{lem:intro_tilt_implies_strong_tilt}
    Let $\mu$ be a log-concave measure on $\R^n$ that is $1$ tilt-stable, then $\mu$ is $Cn^{1/6}\Tilde{K_n}^{1/3}$ strongly tilt-stable, for some universal constant $C>0$.  
\end{lem}
\noindent One ingredient in the proof of Lemma \ref{lem:intro_tilt_implies_strong_tilt} is some sharp analysis of Gaussian perturbations of subgaussian probability measures. Combining Lemmas \ref{lem:intro_strong_tilt_implies_logsob} and \ref{lem:intro_tilt_implies_strong_tilt} we arrive at 
\begin{thm}\label{thm:tilt_logsob}

$\Tilde{G_n} \lesssim n^{1/6}\Tilde{K_n}^{1/3}.$

\end{thm}
\begin{rmk}
Notice that since the Euclidean norm is a $1$-Lipschitz function. By \eqref{eq:log_psi2} we have $\Tilde{K_n} \leq 2\Tilde{G_n}$. Plugging this into Theorem~\ref{thm:tilt_logsob} yields $\Tilde{G_n}\lesssim n^{1/4}$, which is a corollary of Theorem~\ref{thm:MAINLOGSOB}. 
Any improvement over the inequality $\Tilde{K_n} \lesssim n^{1/4}$ would provide a corresponding improvement for $\Tilde{G_n}$.
\end{rmk}

\medskip

We conclude this introduction with some additional comments on the analogy between Conjectures~\ref{conj:KLS} and~\ref{conj:KLSob}. Emmanuel Milman \cite{milman2009role} proved that, for log-concave measures, the Poincaré and log-Sobolev inequalities are respectively equivalent to the (a priori weaker) exponential and Gaussian concentration of Lipschitz functions; see Theorem \ref{thm:milman_logsob} below. Furthermore, it is a well-known consequence of Borell’s lemma that log-concave random variables are $\psi_1$. In particular,
\[
\norm{X\cdot\theta}_{\psi_1}\simeq \norm{X\cdot\theta}_2
\]
whenever $X$ is log-concave and $\theta$ is any vector. Using these two facts, both conjectures may be rewritten as follows:

\begin{conj}
Let $\mu$ be a centered log-concave probability measure. Then
\[
\sup_{F\text{ 1-Lipschitz}} \norm{F-\E_\mu F}_{L^\psi(\mu)}
\;\lesssim\;
\sup_{\abs{\theta}=1}\norm{\langle\,\cdot\,,\theta\rangle}_{L^\psi(\mu)},
\]
where $\psi=\psi_1$ or $\psi=\psi_2$.
\end{conj}

In other words, the concentration properties of Lipschitz functions are conjectured to be, up to a universal constant, the same as those of linear functions. It is of course possible to consider other Young functions, such as $\psi_\alpha(t)=e^{t^\alpha}-1$ for $1<\alpha<2$; see Remark~\ref{rmk_extension}. A significant difference between the KLS case $\alpha=1$ and the others, is that all log-concave measures have $\psi_1$ marginals, but not necessarily $\psi_\alpha$ for $\alpha>1$.

The bounds presented in this paper for Conjecture~\ref{conj:KLSob} match the state of the art on the KLS conjecture prior to Chen’s breakthrough~\cite{chen2021almost}. The best available bound on the KLS constant at the time of writing is $\Gamma_n = O(\sqrt{\log n})$, due to Klartag~\cite{klartag2023logarithmic}. The recent series of spectacular advances on the KLS conjecture (\cite{chen2021almost}, \cite{klartag2022bourgain}, \cite{jambulapati2022slightly}, \cite{klartag2023logarithmic}) all rely on the stochastic localization process introduced by Eldan~\cite{eldan2013thin}, together with a delicate control of the covariance matrix along this process; see \cite{klartag2024isoperimetric} for details. 
In Section~\ref{sec:loc_sto}, we describe an analogous proof scheme to tackle Conjecture~\ref{conj:KLSob} using stochastic localization, and we highlight the key differences between the Poincaré and log-Sobolev settings. Since Conjecture~\ref{conj:KLSob} concerns only a subclass of log-concave probability measures—namely those with subgaussian tails—one would need to establish specific properties of the stochastic localization process starting from a subgaussian log-concave measure. In particular, significant progress on Conjecture~\ref{conj:KLSob} cannot follow from known properties of the stochastic localization process initiated at a general log-concave measure.

    \subsection*{Organization of the paper}
    
    In Section \ref{sec:background}, we recall some backgrounds facts and prove Theorem \ref{thm189}. Section \ref{sec:proof_localization} is devoted to the proof of Theorem \ref{thm:LOCALISATIONLOGSOB} as well as Corollary \ref{cor_secondproof}. In Section \ref{sec:loc_sto}, we investigate an approach to Conjecture \ref{conj:KLSob} via stochastic localization. The only result from this section that we shall use later on is Lemma \ref{lem:intro_strong_tilt_implies_logsob}, which we mentioned in the introduction. Finally, in Section \ref{sec:subclass}, we establish Theorem \ref{thm:logsob_rot_inv} and \ref{thm:tilt_logsob}.
    
    \medskip
    
\textbf{Acknowledgements :} 
The author would like to thank Bo'az Klartag for suggesting the example used to prove the tightness of Proposition~\ref{prop_760}. 
The author is also grateful to the anonymous reviewers for their comments, which greatly helped improve the clarity and presentation of the paper. 
This work was partially supported by the European Research Council (ERC) under the European Union’s Horizon~2020 research and innovation programme, grant agreement No.~101001677 “ISOPERIMETRY”.

%%%%% Compléter intro
\medskip

\section{Background and Proof of Theorem \ref{thm189}}\label{sec:background}
We start by recalling useful facts about subgaussian and subexponential random variables, for which a good reference is \cite{vershynin2018high}, and log-concave vectors. 
\subsection{Subgaussian random variables}
\begin{prop}\label{prop:subgaussian}
Let $X$ be a real random variable. The following properties are equivalent.
\begin{enumerate}
    \item There exists $K_1>0$ such that $\PP(|X|>t)\le 2e^{-t^2/K_1^2}$ for all $t>0$.
    \item There exists $K_2>0$ such that $\E e^{X^2/K_2^2}\le 2$.
    \item There exists $K_3>0$ such that $\E e^{sX}\le 2e^{s^2K_3^2}$ for all $s\in\R$.
\end{enumerate}

Moreover, if $X$ is centered, $\E X=0$, then properties (1)–(3) are also equivalent to

\begin{enumerate}\setcounter{enumi}{3}
    \item There exists $K_4>0$ such that $\E e^{sX}\le e^{s^2K_4^2}$ for all $s\in\R$.
\end{enumerate}

Furthermore, the optimal constants in the inequalities are equivalent up to explicit universal constants.
\end{prop}
\begin{proof}
\((2)\Rightarrow(1)\) by Markov's inequality with $K_1=K_2$,  
\((1)\Rightarrow(2)\) by integration with $K_2=\sqrt{3}K_1$.  
\((2)\Rightarrow(3)\) with $K_3=K_2/2$, since $e^{sx}\le e^{s^2K_2^2/4}e^{x^2/K_2^2}$.  
For \((3)\Rightarrow(2)\), mimicking the proof of \eqref{eq101}: let $Z$ be an independent standard Gaussian,
\[
\E e^{X^2/8K_3^2}=\E e^{XZ/2K_3}\le 2e^{Z^2/4}=2\sqrt{2}\le4.
\]
Hence, by Jensen, $\E e^{X^2/16K_3^2}\le2$, and (2) holds with $K_2\le4K_3$.

For the “Moreover” part, $(4)\!\Rightarrow\!(3)$ is trivial.  
Assume now $\E X=0$ and prove $(3)\!\Rightarrow\!(4)$.  
If $s^2\ge1/(8K_3)$, then $2\le e^{8K_3^2s^2}$ and
\[
\E e^{sX}\le2e^{s^2K_3^2}\le e^{9K_3^2s^2}.
\]
If $s^2\le1/(8K_3)$,
\[
\E e^{sX}\le\E(sX+e^{s^2X^2})
\le(\E e^{X^2/8K_3^2})^{8s^2K_3^2}
\le2^{8K_3^2s^2}=e^{9K_3^2s^2},
\]
using $e^x\le x+e^{x^2}$ and Jensen’s inequality.  
Thus (4) holds with $K_4\le3K_3$.
\end{proof}

If $X$ satisfies any of the above properties, we say that is $X$ is subgaussian, and recall that its $\psi_2$ norm is defined as:
$\norm{X}_{\psi_2} = \inf\left\{t>0 \ / \ \E\left[\exp(X^2/t^2) \leq 2 \right]\right\}$. As explained in the introduction, the Herbst's argument \eqref{eq:gaussian_conc} may be reformulated as follows.
\begin{prop}\label{prop:herbst_psi2}
Let $\mu$ be a probability measure satisfying a log-Sobolev inequality and let f be a centered Lipschitz function. Then 
 $\norm{f}_{L^{\psi_2}(\mu)} \leq 2\rho_{LS} (\mu)\ \abs{f}_{\Lip}$.    
\end{prop}
It is seen, by an application of Jensen's inequality, that the $\psi_2$ norm of a variable controls its $L^2$ norm.
\begin{lem}\label{lem:psi_2_variance} Let $X$ be a subgaussian random variable, then 
$\Var(X) \leq \E X^2 \leq \log(2)\norm{X}_{\psi_2}^2.$
\end{lem}
Using the triangle inequality, we deduce that centering only improves the $\psi_2$ behavior.
\begin{lem}\label{lem:centering_subgauss}
    Let $X$ be a real random variable, then 
    $\left\lVert X - \E X \right\rVert_{\psi_2}
   \le 2 \left\lVert X \right\rVert_{\psi_2}.$

\end{lem}
Finally, we shall need the following standard deviation bound for the norm of a vector with subgaussian marginals.
\begin{prop}\label{prop:norm_subgauss}
    Let $X$ be a random vector in $\R^n$, define $\sigma_{SG}(X) = \sup_{\theta\in\mathbb{S}^{n-1}} \norm{ X\cdot\theta}_{\psi_2}$. Then, there exists a universal constant $c_0>0$ such that for all $t\geq 2c_0\sqrt{n}\sigma_{SG}(X)$
    $$ \PP(\vert X \vert \geq t) \leq \exp\left(-\frac{t^2}{2c_0\sigma^2_{SG}(X)}\right).$$ 
\end{prop}
\begin{proof}
    We use a simple net argument, and we work with sub-optimal constants. Let $\mathcal{N}$ be a $\frac{1}{2}$-net of the sphere. That is a collection of points on the sphere such that any point on the sphere is at distance at most $\frac{1}{2}$ of $\mathcal{N}$. It is classical that we might choose $\mathcal{N}$ such that 
    $$ \vert\mathcal{N}\vert \leq e^{2n}$$
    where in that context $\vert . \vert$ stands for the cardinal of the set. Now, for any $x\in\R^n$, we have that 
    \begin{equation}\label{eq:norm_on_net}
         \vert x\vert \leq 2\sup_{\theta\in\mathcal{N}} x\cdot\theta.
    \end{equation} Now, we use a simple union bound to establish the property.  Let $t\geq 0$,
    \begin{align*}
        \PP(\vert X \vert \geq t) &\leq \PP\left(\exists \theta \in\mathcal{N} \ / \ x\cdot\theta \geq t\right)\\
        &\leq \vert N \vert \exp\left(-t^2/c_0\sigma^2_{SG}(X)\right)\\
        &\leq \exp\left(2n - t^2/c_0\sigma^2_{SG}(X)\right)\\
        &\leq \exp\left(-\frac{t^2}{2c_0\sigma^2_{SG}(X)}\right) \quad\quad\quad \text{for } t\geq 2c_0\sqrt{n}\sigma_{SG}(X).
    \end{align*}
where we chose $c_0\geq 1$.
\end{proof}
\subsection{Subexponential random variables}
\begin{defi}\label{def:subexponential} Let $X$ be a real random variable. We say that if $X$ is subexponential if
there exists $K>0$ such that 
$$\E\exp(\abs{X}/K)\leq 2. $$
In that case, we denote by $\norm{X}_{\psi_1}$ the lowest such $K$.
\end{defi}
Just like the $\psi_2$ norm, the $\psi_1$ norm controls the $L^2$ norm.
\begin{lem}\label{lem:psi_1_variance} Let $X$ be a subexponential random variable, then 
$$\Var(X) \leq \E X^2 \leq \norm{X}_{\psi_1}^2$$
\end{lem}
\begin{proof}
    The lemma follows from the real inequality $1+x^2\leq e^{\abs{x}}$.
\end{proof}
And we deduce the centering lemma
\begin{lem}\label{lem:centering_subexp}
    Let $X$ be a real random variable, then 
    $$\norm{X-\E X}_{\psi_1} \ \leq\ 2\norm{X}_{\psi_1}.$$
\end{lem}

\subsection{Log-concave vectors}
Let $\mu$ be a probability measure on $\R^n$. Its concentration function $\alpha_\mu$ is defined for $r\geq0$ by
$$\alpha_\mu(r) = \sup_{A,\, \mu(A)=\frac{1}{2}}\mu(A_r^c),$$
where $A_r = \{x: d(x,A)\le r\}$ and $d$ denotes the Euclidean distance. It is classical that if $\mu$ satisfies a log-Sobolev inequality with constant $\rho$, then
$$\alpha_\mu(r) \leq \exp(-r^2/\rho^2),$$
which is a reformulation of \eqref{eq:gaussian_conc}.

The next result, due to Milman \cite{milman2009role}, shows that for log-concave vectors, the study of log-Sobolev inequalities can be reduced to the a priori weaker Gaussian concentration.

\begin{thm}\label{thm:milman_logsob}
Let $\mu$ be a log-concave measure, and $K>0$ such that
$$\alpha_\mu(r) \leq \exp(-r^2/K^2).$$
Then
$$\rho_{LS}(\mu) \lesssim K.$$
\end{thm}

\begin{rmk}
In the same line of work, Milman proved that for log-concave probabilities, the Poincaré inequality is quantitatively equivalent to the exponential concentration of Lipschitz functions. This was reproved by Klartag (\cite{klartag2017needle}, Section~5.1) via a new localization lemma which comes with a "guiding" $1$-Lipschitz function. Klartag’s argument can also be adapted to yield an alternative proof of Theorem~\ref{thm:milman_logsob}.
\end{rmk}

The following is the classical Bakry–Émery criterion \cite{bakry2006diffusions}, which provides a quantitative bound on the log-Sobolev constant of strongly log-concave measures.
\begin{thm}\label{thm:bakry_emery}
Let $\mu = e^{-V(x)}\,dx$ be a log-concave probability measure on $\R^n$, and assume $\nabla^2V \succeq tI_n$ for some $t>0$, then
\[
\rho_{LS}^2(\mu) \le \frac{1}{t}.
\]
\end{thm}
Finally, we shall need the following results on one-dimensional log-concave vectors.

\begin{lem}\label{lem:one_dimensional_constants} Let $X$ be a log-concave real random variable with unit variance $\Var(X)=1$. Then there exists universal constants $c_0,c_1,c_2$ such that 
\begin{enumerate}
    \item $\Var(X^2)^{1/2}\geq c_0.$
    \item $\Var((X-\E X)^2)^{1/2} \leq c_1.$
    \item $\norm{X-\E X}_{\psi_1} \leq c_2.$
\end{enumerate}
\end{lem}
\begin{proof}
    The existence of $c_2$ is a reformulation of Borell's lemma, which also implies the existence of $c_1$ :
    $$\Var((X-\E X)^2)\leq \E(X-\E X)^4 \lesssim \Var(X)^2.$$

    Finally, there are various ways of proving the existence of $c_0$. Let $f$ be the density of $X$ and let $a=\E X$. There exists a constant $c>0$, such that 
    \begin{equation}\label{eq_325}
        f(x)\geq c \quad\quad \text{on }[a-c,a+c].
    \end{equation}
    A proof can be found in \cite{eldan2014dimensionality} for example. We have 
    $$\E X^2=a^2+1$$
    Assume without loss of generality that $a\geq0$. By \eqref{eq_325} $X$ is in the interval $[a-1/2,a]$ with probability at least $p=\min(c^2,c/2)>0$. Thus
    $$\Var(X^2) \geq p/2.$$
    
\end{proof}

\subsection{A short proof of Theorem \ref{thm189}}  
Recall that $\Gamma_n$ denotes the KLS constant. The following estimate on the concentration function of log-concave probability measures was proved in \cite{bizeul2022measures}.  

\begin{thm}\label{thm:conc_function_sl}  
Let $\mu$ be a log-concave probability measure with covariance matrix $A$. Then there exists a constant $c_1 > 0$ such that  
\[
\alpha_\mu(r) \leq \exp\left(-c_1 \min\left(\frac{r}{\norm{A}_{op}^{1/2}},\, \frac{r^2}{\norm{A}_{op}\Gamma_n^2\log n}\right)\right).
\]
\end{thm}  

\noindent We are now in position to prove Theorem \ref{thm189}.  

\begin{proof}  
Let $\mu$ be a $1$-subgaussian log-concave probability measure, and let $X$ be distributed according to $\mu$. We assume, without loss of generality, that $\mu$ is centered. Let $A$ be the covariance matrix of $\mu$. By Lemma \ref{lem:psi_2_variance},  
\[
A = \Cov(\mu) \preceq I_n.
\]  
By Emanuel Milman's reduction (Theorem \ref{thm:milman_logsob}), it is enough to estimate the concentration function $\alpha_\mu$. Let $S$ be any set of half measure, $\mu(S) = 1/2$. By Markov’s inequality,  
\[
\PP(\abs{X} \geq 2\sqrt{n}) \ \leq\  \PP(\abs{X} \geq 2\Tr(A)^{1/2}) \ \leq\ \frac{1}{4},
\]
so that $S$ intersects the ball of radius $2\sqrt{n}$. Hence, for any $r \geq 4\sqrt{n}$, we have $S_r^c \subset B(0,r/2)^c$. Using Proposition \ref{prop:norm_subgauss}, we obtain  
\begin{equation}\label{eq:conc_large_r}
    \alpha_\mu(r) \leq \exp\left(-\frac{r^2}{c}\right) \qquad \text{for } r \geq R_0 = c\sqrt{n},
\end{equation}
for some absolute constant $c > 0$.  

For small values of $r$, note that for $r \leq R_0$, we have $r \geq r^2 / R_0$. Plugging this into Theorem \ref{thm:conc_function_sl} yields, for all $r \leq R_0$,  
\begin{align}
    \alpha_\mu(r) &\leq \exp\left(-\frac{c'r^2}{\max\left(\sqrt{n}\norm{A}_{op}^{1/2},\, \norm{A}_{op}\Gamma_n^2\log n \right)}\right)\nonumber\\
    &\leq \exp\left(-\frac{c''r^2}{\sqrt{n}\norm{A}_{op}^{1/2}}\right), \label{eq:4conc_small_r}
\end{align}
where we used the fact that $\Gamma_n^2\log n = O(\sqrt{n})$, which has been known since the breakthrough of Chen \cite{chen2021almost}. 
Combining \eqref{eq:conc_large_r} and \eqref{eq:4conc_small_r} finally yields  
\[
\alpha_\mu(r) \leq \exp\left(-\frac{c'''r^2}{\max(1, \sqrt{n}\norm{A}_{op}^{1/2})}\right)
\]
which concludes the proof.  
\end{proof}

\begin{rmk}\label{rmk_extension}
Using a straightforward generalization of the argument, one may show that for any $1 < \alpha \leq 2$,
\[
\sup_{F\ \text{1-Lipschitz}}
\norm{F - \E_\mu F}_{\psi_\alpha}
\;\lesssim\;
n^{\frac{\alpha - 1}{\alpha}}
\sup_{\abs{\theta} = 1}
\norm{\langle \cdot , \theta \rangle}_{\psi_\alpha},
\]
for all centered log-concave probability measures $\mu$ on $\R^n$.
\end{rmk}

\section{Proof of Theorem \ref{thm:LOCALISATIONLOGSOB}}\label{sec:proof_localization}
In this section, we prove Theorem \ref{thm:LOCALISATIONLOGSOB}. The proof consists in a reduction to dimension one via a localization argument, together with a solution for the one-dimensional case.
\subsection{The one dimensional case}
Our goal is to prove Theorem \ref{thm:LOCALISATIONLOGSOB} in dimension $1$.
\begin{lem}\label{lem:1D_case} Let $X$ be a log-concave random variable on the real line, then
$$\rho_{LS}^2(X) \lesssim \norm{X^2 - \E X^2}_{\psi_1}.$$
\end{lem}
The first step is to show that the right-hand-side is minimized, up to constants, when $X$ is centered:
\begin{lem}\label{lem:centering} Let $Y$ be a centered log-concave random variable on the real line, then
$$\norm{Y^2 - \E Y^2}_{\psi_1} \lesssim \inf_{a\in\R}\norm{(Y+a)^2 - \E (Y+a)^2}_{\psi_1}.$$    
\end{lem}
\begin{proof}
    Let $Y$ be as in the definition. By homogeneity, we may assume that $\E Y^2 = 1$. We temporarily adopt the notation 
    $$K = \norm{Y^2 - \E Y^2}_{\psi_1}.$$
    Let $c_0, c_1, c_2$ be the three constants from Lemma \ref{lem:one_dimensional_constants}. Recall that the standard deviation is a lower bound for the $\psi_1$ norm (Lemma \ref{lem:psi_1_variance}). We distinguish two cases:
    \begin{itemize}
        \item $K \leq 8c_1 c_2$.

        Then, for any $a \in \R$, the log-concave vector $Y+a$ has unit variance. By the first item of Lemma \ref{lem:one_dimensional_constants}, 
        $$\Var((Y+a)^2)^{1/2} \geq c_0.$$
        Thus, 
        \[
        \norm{(Y+a)^2 - \E (Y+a)^2}_{\psi_1}
        \;\geq\; \Var\bigl((Y+a)^2\bigr)^{1/2}
        \;\geq\; c_0
        \;\geq\; \frac{c_0}{8\,c_1 c_2}\,K.
        \]

        \item $K \geq 8c_1 c_2$.
        
        Let $a \in \R$. Expanding the squares, we find that 
        $$\norm{(Y+a)^2 - \E (Y+a)^2}_{\psi_1} = \norm{Y^2 + 2aY - 1}_{\psi_1}.$$
        We compute:
        \begin{align*}
            \Var(Y^2 + 2aY - 1)
            &= \Var(Y^2 + 2aY) \\
            &\geq \left(\Var(2aY)^{1/2} - \Var(Y^2)^{1/2}\right)^2.
        \end{align*}
        Now, since $Y$ is a centered log-concave random variable with unit variance,
        $$c_0 \leq \Var(Y^2)^{1/2} \leq c_1.$$
        Thus, we get 
        \begin{align}
        \norm{Y^2 + 2aY - 1}_{\psi_1}
        &\geq \frac{1}{2}\Var(Y^2 + 2aY - 1)^{1/2} \nonumber\\
        &\geq \frac{1}{2}\left(\Var(2aY)^{1/2} - \Var(Y^2)^{1/2}\right) \nonumber\\
        &\geq \frac{1}{2}\left(2a - c_1\right). \label{eq:bound_a_big}
        \end{align}
        We again make a case disjunction.
        \begin{itemize}
            \item If $a \geq \frac{K}{4c_2} + \frac{c_1}{2}$, we get from \eqref{eq:bound_a_big} that 
            $$\norm{Y^2 + 2aY - 1}_{\psi_1} \geq \frac{K}{4c_2}.$$

            \item If $a \leq \frac{K}{4c_2} + \frac{c_1}{2}$,
            we simply use the triangle inequality:
            \begin{align*}
                \norm{Y^2 + 2aY - 1}_{\psi_1}
                &\geq \norm{Y^2 - 1}_{\psi_1} - \norm{2aY}_{\psi_1} \\
                &= K - 2a \norm{Y}_{\psi_1} \\
                &\geq K - 2a c_2 \\
                &\geq K/2 - c_1 c_2 \\
                &\geq K/4.
            \end{align*}
        \end{itemize}
        In the end, we obtain 
        $$\norm{(Y+a)^2 - \E (Y+a)^2}_{\psi_1} \geq \frac{K}{C},$$
        with $C = \max(4, 4c_2, \frac{c_0}{16c_1 c_2})$, which is the desired result.
    \end{itemize}
\end{proof}
Now we lower bound the quantity $\norm{Y^2 - \E Y^2}_{\psi_1}$ when $Y$ is one dimensional and centered.
\begin{lem}\label{lem:after_centering} Let $Y$ be a centered log-concave random variable on the real line, then
$$\norm{Y^2}_{\psi_1} \lesssim \norm{Y^2 - \E Y^2}_{\psi_1}.$$   
\end{lem}
\begin{proof}
    By the triangle inequality,
    \begin{align*}
        \norm{Y^2}_{\psi_1} &\leq \norm{Y^2 - \E Y^2}_{\psi_1} + \E Y^2 = \norm{Y^2 - \E Y^2}_{\psi_1} + \Var(Y)
    \end{align*}
    Now, applying Lemma \ref{lem:one_dimensional_constants} one more time,
    \begin{align*}
        \Var(Y) &\lesssim \Var(Y^2)^{1/2} \leq \norm{Y^2 - \E Y^2}_{\psi_1},
    \end{align*}which concludes the proof.
\end{proof}

Now we are in position to prove Lemma \ref{lem:1D_case}. Let $X$ be a log-concave real random variable, and let $Y = X-\E X$. Recall Bobkov's result :
$$\rho_{LS}(X) = \rho_{LS}(Y) \lesssim \norm{Y^2}_{\psi_1}.$$
Combining Lemmas \ref{lem:after_centering} and \ref{lem:centering}, we get
\begin{align*}
    \rho_{LS}(X) & \lesssim \norm{Y^2}_{\psi_1} \lesssim \norm{Y^2 - \E Y^2}_{\psi_1} \lesssim \norm{X^2 -\E X^2}_{\psi_1},
\end{align*}
which is what we wanted to prove.
\subsection{A localization argument}
We will need the following geometric version of the localization lemma:
\begin{lem}\label{lem:logsob_geometric_localization_lemma}
Let $\mu$ be a log-concave probability measure on $\R^n$, and let $S$ be any measurable set. 
Then there exists a probability space $(\Omega,\nu)$ and a disintegration 
\[
\mu = \int_{\Omega} \mu_\omega\, d\nu(\omega),
\]
such that $\nu$-almost surely:
\begin{itemize}
    \item $\mu_\omega$ is a one-dimensional log-concave probability measure;
    \item $\mu_\omega(S) = \mu(S)$.
\end{itemize}
\end{lem}
The idea of localization originates in the work of Payne and Weinberger~\cite{payne1960optimal}. Lemma \ref{lem:logsob_geometric_localization_lemma} is in some sense the limiting object of the bisection procedure of \cite{payne1960optimal}.
The formulation of the localization lemma as a disintegration of measure already appears in Gromov-Milman~\cite{gromov1987generalization} in a spherical context. 
The localization method was subsequently developed and applied extensively in convex geometry, but usually stated in a functional form, that is not suited for our purposes
(see, e.g.,~\cite{lovasz1993random,kannan1995isoperimetric}). 
The precise statement of Lemma~\ref{lem:logsob_geometric_localization_lemma} first appears, to the best of our knowledge, 
in Klartag's work~\cite{klartag2017needle}, which also provides additional structural information on the “needles”~$\mu_\omega$ 
that will not be needed here. It is possible to give another proof of this lemma via a version of the stochastic localization process. See for example \cite{eldan2022analysis} for the Gaussian case.
In what follows we will use the compact notation 
\[
\mu = \E_\omega[\mu_\omega]
\]
to denote this disintegration, omitting explicit reference to $\nu$, as it plays no role for us.

For a measure $\mu$ on $\R^n$, we define
\[
\frac{1}{k_\mu}
:= \inf_{\mu(S)\le 1/2}
\frac{\mu^+(S)}{\mu(S)\sqrt{\log\!\frac{1}{\mu(S)}}}
= \inf_{\mu(S)\le 1/2}
\frac{\mu^+(S)}{\mathcal{E}_\mu(S)}\,,
\]
where the infimum runs over all measurable sets $S$ of measure less than $1/2$ and
\[
\mu^+(S) := \liminf_{r\to 0^+}\frac{\mu(S_r)-\mu(S)}{r},
\qquad S_r = \{x\in\R^n:\ \mathrm{d}(x,S)\le r\}.
\]
Ledoux \cite{ledoux1994simple} showed that for all log-concave measure $\mu$,
$$k_\mu \simeq \rho_{LS}(\mu).$$
Now, we fix a set $S$ and a disintegration $\mu=\E_\omega[\mu_\omega]$ given by the localization lemma. We may write
\begin{align*}
    \mu^+(S) &= \E_\omega[\mu_\omega^+(S)]\\
             &\geq \E_\omega\!\left[\frac{1}{k_{\mu_\omega}}\,\mathcal{E}_{\mu_\omega}(S)\right] \\
             &= \E_\omega\!\left[\frac{1}{k_{\mu_\omega}}\right] \mathcal{E}_{\mu}(S).
\end{align*}
Thus, we need to estimate $\frac{1}{k_{\mu_\omega}}$. 
Let $a \ge 0$ and denote $K_a = \left\lVert |X|^2 - a \right\rVert_{\psi_1}$, where $X \sim \mu$. 
Write $X_\omega$ for a random vector with law $\mu_\omega$. We have
\[
2 \;\ge\; \E e^{(|X|^2 - a)/K_a}
   \;=\; \E_\omega\!\left[\E e^{(|X_\omega|^2 - a)/K_a}\right].
\]
By Markov's inequality, with probability at least $1/2$,
\[
\E e^{(|X_\omega|^2 - a)/K_a} \le 4.
\]
We work on this event, denoted by $\mathcal{U}$. Since each $X_\omega$ is one-dimensional, we may write $X_\omega = b_\omega + \xi_\omega \theta_\omega$, where $\xi_\omega$ is log-concave, $|\theta_\omega| = 1$, and $b_\omega \perp \theta_\omega$. Thus,
\[
\E e^{(|X_\omega|^2 - a)/K_a} = \E e^{(|b_\omega|^2 + \xi_\omega^2 - a)/K_a} \le 4.
\]
That is, we have shown that
\[
\left\lVert \xi_\omega^2 - (a - |b_\omega|^2) \right\rVert_{\psi_1} \le 2K_a.
\]
In particular,
\begin{equation}\label{eq:xi_omega_psi1}
    \left\lVert \xi_\omega^2 - \E \xi_\omega^2 \right\rVert_{\psi_1} \lesssim K_a,
\end{equation}
since the mean minimizes the $\psi_1$ norm up to a universal constant. Using Lemma~\ref{lem:1D_case}, we may rewrite~\eqref{eq:xi_omega_psi1} as
\begin{equation}\label{eq:k_xi_omega}
    \frac{1}{k_{\mu_\omega}} \gtrsim \frac{1}{K_a}.
\end{equation}
Finally, putting everything together, we obtain
\begin{align*}
    \mu^+(S)
    &\ge \E_\omega\!\left[\frac{1}{k_{\mu_\omega}}\right] \mathcal{E}_{\mu}(S)\\
    &\ge \E_\omega\!\left[\frac{1}{k_{\mu_\omega}}\,1_{\mathcal{U}}\right] \mathcal{E}_{\mu}(S)\\
    &\gtrsim \frac{1}{K_a}\,\mathcal{E}_{\mu}(S).
\end{align*}
Since $S$ is arbitrary, we conclude that
\[
k_\mu \lesssim K_a = \left\lVert |X|^2 - a \right\rVert_{\psi_1}.
\]
Taking $a = \E |X|^2$ completes the proof.
\subsection{Proof of Corollary \ref{cor_secondproof}}
To conclude this section, we explain how Theorem~\ref{thm:LOCALISATIONLOGSOB} allows us to recover 
Theorem~\ref{thm:MAINLOGSOB} up to a logarithmic factor, by proving Corollary~\ref{cor_secondproof}. Let $X$ be a $1$-subgaussian log-concave random vector in $\R^n$. In view of Theorem~\ref{thm:LOCALISATIONLOGSOB}, Corollary~\ref{cor_secondproof} will follow if we show that  
\begin{equation}\label{eq849}
    \bigl\lVert\,|X|^2- \E\abs{X}^2\,\bigr\rVert_{\psi_1} \lesssim \sqrt{n\log n}.
\end{equation}

\paragraph{Step 1}
We begin with a preprocessing of $X$. Since $X$ is $1$-subgaussian, we have $\Cov(X)\preceq(\log 2)\,I_n$. Let $G$ be a Gaussian vector, independent of $X$, with covariance $\Sigma = I_n - \Cov(X)$, and set $\widetilde X = X+G$. Then $\widetilde X$ is isotropic, log-concave, and $2$-subgaussian. Moreover, by the triangle inequality,
\[
\bigl\|\,|X|^2-\E|X|^2\,\bigr\|_{\psi_1}
\;\leq\;
\bigl\|\,|\widetilde X|^2-\E|\widetilde X|^2\,\bigr\|_{\psi_1}
+ \bigl\|\,|G|^2-\E|G|^2\,\bigr\|_{\psi_1}
+ 2\bigl\|\,X\!\cdot\!G\bigr\|_{\psi_1}.
\]

Since $\Sigma\preceq I_n$, one easily checks that 
\[
\bigl\|\,|G|^2-\E|G|^2\,\bigr\|_{\psi_1}\lesssim\sqrt{n}.
\]
For instance, arguing as in \eqref{eq243}
\[
\bigl\lVert\,|G|^2- \E\abs{G}^2\,\bigr\rVert_{\psi_1}
    \;\lesssim\;
    \bigl\|\,|G| - \E|G|\,\bigr\|_{\psi_2}^2
    \;+\;
    \E|G| \,\cdot\,
    \bigl\|\,|G| - \E|G|\,\bigr\|_{\psi_1},
\]
together with $c_P(G)\leq \rho_{LS}(G)\leq 1$ or compute directly, which is elementary.  
On the other hand, letting $K=\|X\|_{\psi_2}$,
\[
\E e^{\frac{X\cdot G}{K}}
= \E e^{\frac{|X|^2}{K^2}}
\leq 2,
\]
and Proposition~\ref{prop:norm_subgauss} yields $\|X\!\cdot\!G\|_{\psi_1} \leq \||X|\|_{\psi_2} \lesssim \sqrt{n}$.

We conclude that 
\begin{equation}\label{eq874}
\bigl\|\,|X|^2-\E|X|^2\,\bigr\|_{\psi_1}
\;\leq\;
\bigl\|\,|\widetilde X|^2-\E|\widetilde X|^2\,\bigr\|_{\psi_1}
+ O(\sqrt{n}).
\end{equation}
Thus, in order to prove \eqref{eq849}, it is enough to bound $\bigl\|\,|\widetilde X|^2-\E|\widetilde X|^2\,\bigr\|_{\psi_1}$ when $\widetilde X$ is isotropic and $2$-subgaussian.

\paragraph{Step 2}
Using the triangle inequality, as in \eqref{eq243}, we decompose
\[
\bigl\lVert\,|\widetilde X|^2- \E\abs{\widetilde X}^2\,\bigr\rVert_{\psi_1}
    \;\lesssim\;
    \bigl\|\,|\widetilde X| - \E|\widetilde X|\,\bigr\|_{\psi_2}^2
    \;+\;
    \E|\widetilde X| \,\cdot\,
    \bigl\|\,|\widetilde X| - \E|\widetilde X|\,\bigr\|_{\psi_1}.
\]
We bound the two terms appearing on the right-hand side separately. For the first term, using Proposition~\ref{prop:norm_subgauss}, for $t\geq C\sqrt{n}$,
\[
\PP\bigl(|\widetilde X| - \E|\widetilde X| \geq t\bigr) 
\leq \PP\bigl(|\widetilde X|\geq t\bigr) 
\leq \exp(-ct^2),
\]
where $C>1$ and $c>0$ are universal constants, and where we used that $\E|\widetilde X|\leq \sqrt{n}$ since $\widetilde X$ is isotropic. For $t\leq C\sqrt{n}$, we use a deviation result of Guédon and Milman~\cite[Theorem~1.1]{guedon2011interpolating}, which implies that 
\[
\PP\bigl(|\widetilde X| - \E|\widetilde X| \geq t\bigr) \leq C'\exp(-c' t^2/\sqrt{n})
\]
for some universal constants $C',c'>0$. Putting both bounds together, we arrive at
\begin{equation}\label{eq_tojustify}
        \bigl\|\,|\widetilde X| - \E|\widetilde X|\,\bigr\|_{\psi_2}^2
    \;\lesssim\;
    \sqrt{n}.
\end{equation}

For the second term, we need to control
$\bigl\|\,|\widetilde X| - \E|\widetilde X|\,\bigr\|_{\psi_1}$.
Since the Euclidean norm is $1$-Lipschitz, it follows that
\[
    \bigl\|\,|\widetilde X| - \E|\widetilde X|\,\bigr\|_{\psi_1}
    \;\lesssim\;
    C_P(\widetilde X)
    \;\lesssim\;
    \sqrt{\log n},
\]
where in the last inequality we used the best known bound on the KLS constant 
due to Klartag~\cite{klartag2023logarithmic}. We conclude that 
\[
\bigl\|\,|\widetilde X|^2-\E|\widetilde X|^2\,\bigr\|_{\psi_1}\lesssim\sqrt{n\log n}.
\]
Thus we have proved Corollary~\ref{cor_secondproof}, as 
\[
\rho_{LS}^2(X)\lesssim\bigl\lVert\,|X|^2- \E\abs{X}^2\,\bigr\rVert_{\psi_1} 
\lesssim \bigl\lVert\,|\widetilde X|^2- \E\abs{\widetilde X}^2\,\bigr\rVert_{\psi_1} + O(\sqrt{n}) 
\lesssim\sqrt{n\log n}.
\]

The logarithmic factor would disappear if one could show that
\[
    \bigl\|\,|\widetilde X| - \E|\widetilde X|\,\bigr\|_{\psi_1} = O(1)
\]
for every isotropic log-concave random vector $\widetilde X$. 
This conjectured concentration estimate lies between the KLS conjecture and the thin-shell conjecture. 
The latter was recently resolved by Klartag and Lehec~\cite{klartag2025thin}, 
following a breakthrough of Guan~\cite{guan2024note}, 
and asserts that $\Var(|\widetilde X|) = O(1)$ for any isotropic log-concave vector $\widetilde X$. 
Furthermore, using reverse Hölder inequalities for polynomials, the authors of 
\cite{klartag2025thin} show that one can obtain a dimension-free bound 
on the $\psi_{1/2}$-norm of $|\widetilde X| - \E|\widetilde X|$. 
However, this argument does not yield a corresponding bound for the $\psi_1$-norm.

\section{An approach via stochastic localization}\label{sec:loc_sto}
In this section we describe a general strategy to estimate the log-Sobolev constant of a log-concave probability measure using stochastic localization. We briefly recall the definition and basic properties of the process in the Lee-Vempala formulation \cite{lee2017eldan}, we refer to the lecture notes \cite{klartag2024isoperimetric} for a detailed exposition.

Let $\mu$ be a log-concave probability measure on $\R^n$ with density $f$. 
For $t \ge 0$ and $h \in \R^n$, we define a probability measure $\mu_{t,h}$ by
\begin{equation}\label{eq:def_mu_t_h}
    \frac{d\mu_{t,h}(x)}{d\mu(x)} 
    = \frac{1}{Z_{t,h}}\, e^{-t|x|^2 + h\cdot x}, 
    \qquad x \in \R^n,
\end{equation}
where $Z_{t,h}$ a the normalizing constant. 
We denote its density by $f_{t,h}$ and set
\[
a_{t,h} := \int_{\R^n} x \, d\mu_{t,h}(x)
\qquad\text{and}\qquad
A_{t,h} := \int_{\R^n} (x - a_{t,h})(x - a_{t,h})^{\!T} \, d\mu_{t,h}(x),
\]
for the barycenter and covariance matrix of $\mu_{t,h}$. Consider the stochastic differential equation:
\begin{equation}\label{eq:dh_t}
    h_0=0 \quad dh_t = a_{t,h_t}dt + dB_t
\end{equation}
where $(B_t)_{t\geq0}$ is a standard Brownian motion. The stochastic localization of $\mu$ is the measure-valued process $(\mu_{t,h_t})_{t\geq0}$, which by a slight abuse of notations, we hereby denote by $(\mu_t)_{t\geq0}$. Accordingly, we denote by $f_t, a_t$ and $A_t$ the density, barycenter, and covariance matrix of $\mu_t$. The following lemma is classical, and is an alternate definition of the process.
\begin{lem} For any $x\in\R^n$, 
\begin{equation}\label{eq873}
    df_t(x) = (x-a_t)f_t(x)\cdot dB_t
\end{equation}
\end{lem}
As an immediate consequence, we obtain 
\begin{lem}\label{lem:martingale}
    For any test function $\varphi$, the process $\left(\int_{\R^n} \varphi d\mu_t\right)_{t \geq0} $ is a martingale.
\end{lem}
Since $\mu$ is log-concave, it is clear from \eqref{eq:def_mu_t_h} that the random probability measure $\mu_t = \mu_{t,h_t}$ is $t$-strongly log-concave. Thus, by the Bakry-Emery criterion, we have the deterministic bound
\begin{equation}\label{eq852}
    \rho_{LS}(\mu_t)\leq \frac{1}{t}.
\end{equation}
Thus, our goal is to relate $\rho_{LS}(\mu)$ to $\rho_{LS}(\mu_T)$ for a sufficiently large $T>0$. Let $g$ be a smooth function, and let $M_t = \int g^2d\mu_t$. By Lemma \ref{lem:martingale}, $(M_t)_{t\geq0}$ is a martingale. We obtain the following decomposition of entropy.
\begin{lem}\label{lem:entropy_decomposition}
For all $T>0$,
\begin{align}
    \Ent_\mu(g^2)
    &= \E\Ent_{\mu_T}(g^2) + \Ent(M_T)\nonumber\\
    &= \E\Ent_{\mu_T}(g^2) + \E\int_0^T \frac{d[M]_t}{2M_t}\nonumber\\
    &\le \frac{2}{T}\E_\mu(|\nabla g|^2) + \E\int_0^T \frac{d[M]_t}{2M_t}.\label{eq:entropy_decomposition}
\end{align}
\end{lem}

\begin{proof}
The first identity follows from the martingale property (Lemma~\ref{lem:martingale}),
the second from straightforward Itô calculus.
Finally, the last inequality uses \eqref{eq852} together with the martingale property once again.
\end{proof}

In view of \eqref{eq:entropy_decomposition}, we aim to bound the term
$\E\int_0^T \frac{d[M]_t}{2M_t}$.
Since one cannot expect a log-Sobolev inequality to hold for an arbitrary log-concave measure~$\mu$,
the bound should involve, in some way, the subgaussian parameter~$\sigma_{SG}^2(\mu)$.
For cosmetic reasons, we introduce an equivalent quantity that will be more convenient for us:
for a measure~$\nu$ with barycenter~$b$, define
\[
\Tilde{\sigma}(\nu)
= \inf\bigl\{K>0 : \forall u\in\R^n,\ 
\E_{X\sim\nu} e^{(X-b)\cdot u} \le e^{u^2K^2/2}\bigr\}.
\]
By the “Moreover” part of Proposition~\ref{prop:subgaussian}, we have
$\Tilde{\sigma} \simeq \sigma_{SG}$.
In the following, we write $\Tilde{\sigma_t}=\Tilde{\sigma}(\mu_t)$.
\begin{lem}\label{lem:dM_t_control}
\[
\frac{d[M]_t}{2M_t} \le \Tilde{\sigma_t}^2\,\Ent_{\mu_t}(g^2).
\]
\end{lem}
\begin{proof}
    Let $t\geq0$, using \eqref{eq873} we first compute
    \begin{align*}
        dM_t = d\int g^2d\mu_t = \left(\int g^2(x-a_t)\mu_t\right)\cdot dB_t.
    \end{align*}
    Let $\lambda>0$ a parameter to be determined later. From the previous computation,
    \begin{align*}
        d[M]_t &= \abs{\int g^2(x-a_t)\mu_t}^2\\
        &= \sup_{\theta \in \mathcal{S}^{n-1}}\left(\int g^2(x-a_t)\cdot\theta\mu_t\right)^2\\
        &= \lambda M_t^2 \sup_{\theta \in \mathcal{S}^{n-1}}\left(\int \frac{g^2}{M_t}\frac{(x-a_t)\cdot\theta}{\sqrt{\lambda}}\mu_t\right)\\
        &\leq \lambda M_t^2 \sup_{\theta \in \mathcal{S}^{n-1}}\left(\Ent_{\mu_t}\left(\frac{g^2}{M_t}\right) + \log \E_{\mu_t}e^{\frac{(x-a_t)\cdot\theta}{\sqrt{\lambda}}}\right)^2\\
        &\leq \lambda M_t^2\left(\frac{1}{M_t}\Ent_{\mu_t}\left(g^2\right) + \frac{\Tilde{\sigma_t}^2}{2\lambda}\right)^2\\
        &\leq 2\left(\lambda\Ent_{\mu_t}^2(g^2) + \frac{\Tilde{\sigma_t}^4M_t^2}{4\lambda}\right)
    \end{align*}
    where in the first inequality we used Gibb's variational principle and in the last, we used that $(a+b)^2\leq 2(a^2 + b^2)$ for reals $a,b$. Finally, we get that for any $\lambda>0$, 
    $$\frac{d[M]_t}{2M_t} \leq \frac{\lambda\Ent_{\mu_t}^2(g^2)}{M_t} + \frac{\Tilde{\sigma_t}^4M_t}{4\lambda}.$$
    Choosing the optimal $\lambda = \frac{\Tilde{\sigma_t}^2M_t}{2\Ent_{\mu_t}(g)}$ concludes the proof.
\end{proof}
We deduce that a deterministic bound on $\Tilde{\sigma_t}$ yields a bound on $\rho_{LS}(\mu)$.
\begin{lem}\label{lem940}
    Let $\mu$ be a log-concave probability measure. Assume that for all $t\geq0$, the bound $\sigma_t\leq\beta$ holds almost surely, then $\rho_{LS}(\mu) \leq 2\beta$.
\end{lem}
\begin{proof}
    By plugging the bound on $\sigma_t$ into Lemma \ref{lem:dM_t_control} and Lemma \ref{lem:entropy_decomposition} we get, for an arbitrary smooth function $g$, and $T>0$:
    \begin{align*}
        \Ent_{\mu}(g^2) &\leq \frac{2}{T}\E_\mu(\abs{\nabla g}^2) + \beta^2 \int_0^T \E \Ent_{\mu_t}(g^2) dt\\
        &\leq \frac{2}{T}\E_\mu(\abs{\nabla g}^2) + \beta^2T \Ent_{\mu}(g^2)\\
    \end{align*}
    Choosing $T=\frac{1}{2\beta^2}$ yields 
    $$\Ent_\mu(g^2) \leq 8\beta^2\E_\mu(\abs{\nabla g}^2).$$
    Equivalently, $\rho_{LS}(\mu)^2\leq 4\beta^2$, which is the desired result.
\end{proof}

\begin{rmk}
    Since $\mu_t$ is $t$-strongly log-concave we always have the deterministic bound $\Tilde{\sigma_t}\leq \rho_{LS}(\mu_t)\leq\frac{1}{t}$ almost surely. Thus the hypothesis of Lemma \ref{lem940} is essentially about small times $t\geq0$.
\end{rmk}

We are now in position to bound the log-Sobolev constant of strongly tilt stable log-concave probability measures as claimed in the introduction. It is a direct application of the previous lemma.
\begin{proof}[Proof of Lemma \ref{lem:intro_strong_tilt_implies_logsob}]
    Let $\mu$ be a $\beta$-strongly tilt stable log-concave probability. That is
    \begin{equation}\label{eq_601}
    \sup_{t>0, \ h \in\R^n}\ \norm{\Cov(\frac{1}{Z_{h,t}}\mu e^{-t\abs{x}^2 + h\cdot x})}_{op} \leq \beta^2.
    \end{equation}
    Fix $t_0>0$ and let $h_0\in\R^n$. Equation \eqref{eq_601} implies that $\mu_{t_0,h_0}$ is $\beta$-strongly tilt-stable too. In particular, it is $\beta$-tilt stable. From the proof of Lemma \ref{lem315}, we know that tilt-stability implies a quadratic bound on the log-Laplace \eqref{eq320}. Namely, $\Tilde{\sigma}(\mu_{t,h_0}) \leq \beta$. By letting $t_0$ and $h_0$ take arbitrary values, we see that
    $$\Tilde{\sigma_t} \leq \beta \quad \text{a.s.}$$
    It follows that $\rho_{LS}(\mu)\leq 2\beta$ by Lemma \ref{lem940}.
\end{proof}
\medskip

In the rest of this section, we draw a parallel with the KLS conjecture, where stochastic localization techniques have yielded a logarithmic bound on the Poincaré constant of isotropic log-concave vectors. We only sketch the proofs, as our main purpose is to illustrate the similarities and differences between the log-Sobolev and Poincaré settings. We refer the reader to~\cite{klartag2024isoperimetric} for details.

In the context of the KLS conjecture one seeks to bound the variance of an arbitrary function smooth $\varphi$. We denote by $N_t$ the martingale $N_t = \int_{\R^n}\varphi d\mu_t$. Much like we decomposed the entropy (Lemma \ref{lem:entropy_decomposition}), we can decompose the variance of $\varphi$ as follows :
\begin{align}
        \text{For all } T\geq 0 \quad \Var_{\mu}(\varphi) & = \E\Var_{\mu_T}(\varphi) + \Var(N_T)  \nonumber\\
        &\leq \frac{1}{T}\E_{\mu}\left(\abs{\nabla \varphi}^2\right) + \E \int_{0}^{T}d[N]_t\label{eq:variance_decomposition}
\end{align}
where we used that $C_P(\mu_t)\leq \rho_{LS}(\mu_t) \leq \frac{1}{t}$ almost surely. Thus, instead of a bound on $\frac{d[M]_t}{2M_t}$ as in Lemma \ref{lem:dM_t_control}, here we simply need a bound on the quadratic variation $[N]_t$. Using the equation
$$dN_t = \left(\int (x-a_t)\varphi(x)d\mu_t\right)\cdot dB_t$$
and the Cauchy-Schwarz inequality, one gets
\begin{equation}\label{eq:dM_t_control_variance}
    d[N]_t \leq \norm{A_t}_{op}\Var_{\mu_t}(\varphi).
\end{equation}
According to a refinement of a theorem of E. Milman \cite{milman2009role}, proved in \cite{klartag2024isoperimetric}, for any log-concave measure $\mu$, there exists a 1-Lipschitz function $\varphi_0$ satisfying 
\begin{equation}
    \Var_\mu(\varphi_0) \simeq C_P(\mu)^2 \simeq \norm{\varphi_0}_\infty^2.
\end{equation}
Using the decomposition of variance \eqref{eq:variance_decomposition} as well as the bound \eqref{eq:dM_t_control_variance}, we get
\begin{align*}
    \frac{1}{C_1}C_P^2(\mu)\leq\Var_\mu(\varphi_0) &\leq \frac{1}{T}\E_{\mu}\left(\abs{\nabla \varphi_0}^2\right) + \E \int_{0}^{T}d[N]_t \\
    &\leq \frac{1}{T} + \E \int_{0}^T \norm{A_t}_{op}\Var_{\mu_t}(\varphi_0)\\
    &\leq \frac{1}{T} + \left(\int_0^T\E\norm{A_t}_{op}\right)\norm{\varphi_0}_\infty^2\\
    &\leq \frac{1}{T} + C_2\left(\int_0^T\E\norm{A_t}_{op}\right)C_P^2(\mu),
\end{align*}
where $C_1,C_2$ are universal constants. We deduce the following prototypical Lemma, relating the Poincaré constant of $\mu$ to the \textit{expected} operator norm of the covariance process $(A_t)_{t\geq0}$.
\begin{lem}\label{lem_KLS} Let $T>0$ such that $\left(\int_0^T\E\norm{A_t}_{op}\right)\leq c_3$, then
$$C_P^2(\mu) \leq \frac{2C_1}{T}$$
where $c_3 = \frac{1}{2C_1C_2}$.
\end{lem}
In other words, it is enough to establish a bound in expectation on $\norm{A_t}_{op}$. The exposition presented here is taken from \cite{klartag2024isoperimetric}. The fact that an almost extremizer $\varphi_0$ of the Poincaré constant can be chosen bounded is not essential: similar lemmas can be obtained by relying only on the $1$-Lipschitz property of $\varphi_0$ (see e.g. \cite{klartag2023logarithmic}). Alternatively, one may work with the equivalent Cheeger inequality instead of the Poincaré inequality, using that the isoperimetric profile is maximized for sets of half measure (see e.g. \cite{eldan2013thin, lee2017eldan}). However, these reductions are crucial to allow the use of a bound in expectation (or with high probability) on the covariance matrix $A_t$. A direct proof working with an arbitrary smooth function $u$ with $\int \abs{\nabla u}^2 = 1$ is not known. More precisely, working with such a function yields
\[
\Var_\mu(u) \le \frac{1}{T} + \E\int_{0}^T \norm{A_t}_{op}\Var_{\mu_t}(u)\,dt.
\]
From the variance decomposition \eqref{eq:variance_decomposition}, one always has $\E\Var_{\mu_t}(u)\le \Var_{\mu}(u)$, so if $\norm{A_t}_{op}$ is bounded almost surely, the argument goes through. But when only a high-probability bound is available, there is no clear way to control the low-probability event on which $A_t$ may blow up.

It is useful to compare Lemma \ref{lem_KLS} and Lemma \ref{lem940}. The main drawback of Lemma \ref{lem940} is that it requires a deterministic bound on $\Tilde{\sigma_t}$ instead of a bound with high probability or in expectation.
This, in our view, constitutes the first bottleneck toward sharper bounds on the log-Sobolev constant $G_n$. In the log-Sobolev setting, one works with an arbitrary smooth function $g$ and arrives at
\[
\Ent_{\mu}(g^2) \le \frac{2}{T}\E_\mu(\abs{\nabla g}^2) + \E \int_0^T \Tilde{\sigma_t}^2 \Ent_{\mu_t}(g^2)\,dt.
\]
Since $g$ is arbitrary, it seems impossible to rely on anything weaker than an almost-sure bound on $\Tilde{\sigma_t^{\,2}}$---a condition that generally fails. E.~Milman showed that Gaussian concentration of Lipschitz functions suffices to imply a log-Sobolev inequality for $\mu$, but despite our efforts we were unable to leverage this reduction to relax the almost-sure requirement on $\Tilde{\sigma_t}$ in Lemma~\ref{lem940}.

Finally, even if the hypothesis of Lemma~\ref{lem940} could be weakened to a bound in expectation or with high probability on $\sigma_t$, controlling $\sigma_t$ along the stochastic localization process (starting from $\sigma_0=1$, say) remains a highly challenging problem.

\section{Some subclasses of subgaussian log-concave probabilities}\label{sec:subclass}
\subsection{Rotationally invariant measures}
We say that a measure $\mu$ is rotationally invariant if for any orthogonal transformation $R\in O(n)$ and any measurable set $A$, $\mu(RA)=\mu(A)$. When $\mu$ is absolutely continuous with respect to the Lebesgue measure then $d\mu = \lambda(\abs{x})dx$ for some positive integrable function $\lambda$. Now, it is easy to check that $\mu$ is log-concave if and only if $\lambda$ is log-concave and nonincreasing.

In order for $\mu$ to satisfy a log-Sobolev inequality, its marginals must be subgaussian. This is not always the case, as one might consider a density proportional to $e^{-\abs{x}}$. By scaling, we assume that $\mu$ is \textit{isotropic}, that is for any $\theta\in\mathcal{S}^{n-1}, \ \int_{\R^n}\left(x\cdot\theta\right)^2 d\mu=1$. Equivalently,
$$\int_{\R^n}\abs{x}^2d\mu = n.$$ Bobkov \cite{bobkov2010gaussian} established the following estimate for the concentration function.
\begin{thm}
The concentration function of $\mu$ satisfies :
$$ \alpha_{\mu}(r) \leq e^{-cr^2} \quad \text{for }r\leq \sqrt{n}$$
for some universal constant $c>0$.
\end{thm}
Let $K\geq 1$. Since $\alpha_\mu$ is non-increasing, we may trivially extrapolate the above estimate to the range $r\in[0,K\sqrt{n}]$ as
\begin{equation}\label{eq:rotation_invariant_small_r}
    \alpha_{\mu}(r)\leq \exp\!\left(-\frac{c r^2}{K^2}\right)
    \qquad\text{for } r\leq K\sqrt{n}.
\end{equation}
For large values of $r$, we argue as in the proof of Theorem~\ref{thm189}. Let $A$ be a set with $\mu(A)=1/2$. Since $\mu$ is isotropic, $\mu(B(0,2\sqrt{n}))\geq \frac{3}{4}$ by Markov's inequality. Consequently, for $r>4\sqrt{n}$, we have the inclusion
$$
B(0,r/2)\subset B(0,r-2\sqrt{n})\subset A_r.
$$
Thus, using Proposition~\ref{prop:norm_subgauss}, we conclude that for
$r\geq 4\max(1,c_0)\,\sigma_{SG}(\mu)\sqrt{n},$ 
\begin{equation}\label{eq1036}
    \mu(A_r^c)\leq \mu(B(0,r/2)^c)
    \leq \exp\!\left(-\frac{r^2}{8c_0\,\sigma_{SG}^2(\mu)}\right),
\end{equation}
where $c_0$ is the constant appearing in Proposition~\ref{prop:norm_subgauss}, and we used the fact that $\sigma_{SG}(\mu)\geq 1$ since $\mu$ is isotropic. Combining \eqref{eq:rotation_invariant_small_r} with $K=4\max(1,c_0)\,\sigma_{SG}(\mu)\sqrt{n}$ and \eqref{eq1036}, we arrive at
$$
\alpha_\mu(r)\leq \exp\!\left(-\frac{c_1 r^2}{\sigma_{SG}(\mu)^2}\right)
\qquad\text{for all } r\geq 0,
$$
for some universal constant $c_1>0$, which is equivalent to Theorem~\ref{thm:logsob_rot_inv}.

\subsection{Tilt-stable measures}\label{subsection:tilt}
The goal of this section is to prove that tilt-stable log-concave measures are in fact strongly tilt-stable, albeit with a dimension dependent constant. More precisely, we prove Lemma \ref{lem:intro_tilt_implies_strong_tilt}. Since strongly-tilt stable log-concave measures enjoy a nice bound on their log-Sobolev constant (Lemma \ref{lem:intro_tilt_implies_strong_tilt}), this yields a proof of Theorem \ref{thm:tilt_logsob}. We begin by recalling a few definitions and facts
Let $\nu$ be a probability measure on $\R^n$. For $t\geq0$ and $h\in\R^n$, recall that $\nu_{t,h}$ is the Gaussian tilt of $\nu$ :
\[
\frac{d\nu_{t,h}}{d\nu}(x)
    = \exp\!\left(h\cdot x - \tfrac{t}{2}|x|^2 - L_t(h)\right).
\]
where
\[
L_t(h)
    = \log \int_{\R^n} \exp\!\left(h\cdot x - \tfrac{t}{2}|x|^2\right)\, d\nu(x),
\]
Recall that $\nu$ is said to be $\beta$-tilt stable if 
$$\nabla_h^2 L_0(h) = \Cov(\mu_{0,h}) \preceq \beta^2 I_n \qquad \text{for all }h\in\R^n.$$
Furthermore, $\nu$ is $\beta$-strongly tilt stable if 
$$\nabla_h^2 L_t(h) = \Cov(\mu_{t,h}) \preceq \beta^2 I_n \qquad \text{for all }h\in\R^n\text{ and }t\geq0.$$
Finally, we recall the notation for the tilt operator $\tau_h \nu = \nu_{0,h}$

Tilted measures and the log-Laplace transform are known to play a central role in convex geometry (\cite{klartag2006convex}, \cite{klartag2012centroid}, \cite{eldan2011approximately}). In the context of the discrete hypercube, tilt-stable measures were studied by Eldan and Shamir \cite{eldan2022log}, where they are shown to exhibit non-trivial concentration. They also appear in Eldan and Chen \cite{chen2022localization}. 

Examples of tilt-stable measures include strongly log-concave measures, as well as products of tilt-stable measures. Indeed, if $\nu=\nu_1\otimes\dots\otimes\nu_k$ is a product measure and $h=(h_1,\dots,h_k)$, then
$$
\tau_h\nu=\tau_{h_1}\nu_1\otimes\dots\otimes\tau_{h_k}\nu_k,
$$
so if each component is $\beta_i$ tilt-stable, the product $\nu$ is itself tilt-stable with constant $\beta=\max_{1\le i\le k}\beta_i$. In particular, the uniform measure on the discrete or continuous hypercube is tilt-stable.

Klartag~\cite{klartag2014concentration} also constructed an interesting class of strongly tilt-stable measures supported on the hypercube. Namely, if $\mu=e^{-V(x)}dx$ is a log-concave probability measure satisfying $\partial_{ii}V\le 1$ for all $1\leq i\leq n$, uniformly on the unit hypercube, then $\mu$ is $C$-strongly tilt-stable for some universal constant $C>0$. Such measures can be obtained by convolving any log-concave density with a standard Gaussian and then restricting the result to the hypercube.
It would be interesting to provide a direct and elementary proof of the strong tilt-stability of these measures. More generally, it would be desirable to find general sufficient conditions ensuring that a log-concave probability measure $\mu$ is (strongly) tilt-stable.

%\com{Commentaires sur les tilts-stables, et lemme de Bodineau-Bauerschmidt généralisé ?}

 In the following, given a tilt-stable log-concave probability measure $\mu$, we use a perturbation argument to show that it is in fact strongly tilt-stable. Our goal is to estimate
$$
\sup_{t>0,\ h\in\R^n}\ \norm{\Cov(\mu_{t,h})}_{op}.
$$
The key idea is that every Gaussian tilt $\mu_{t,h}$ can be obtained as a centered Gaussian perturbation of a linear tilt $\mu_{0,h_0}$. This nontrivial reduction is justified by the next lemma.

\begin{lem}\label{lem:get_rid_of_tilt}
    Let $\mu$ be a probability measure such that all exponential tilts of $\mu$ are well defined. Fix $h\in\R^n$ and $t>0$. Then there exists $h_0\in\R^n$ such that
    $$
        \mu_{t,h}
        = \frac{1}{Z}\,(\tau_{h_0}\mu)\,e^{-t\abs{x-\bary(\tau_{h_0}\mu)}^2},
    $$
    where $Z$ is a normalizing constant. In other words, all measures $\mu_{t,h}$ arise as \textbf{centered} Gaussian perturbations of linear tilts of $\mu$.
\end{lem}

\begin{proof}
    Developing the right-hand side gives
    $$
        \frac{1}{Z}\,(\tau_{h_0}\mu)\,e^{-t\abs{x-\bary(\tau_{h_0}\mu)}^2}
        = \mu_{t,\,h_0+2t\,\bary(\tau_{h_0}\mu)}.
    $$
    Thus we must show that, for each $t>0$, the map
    $$
        F:\ h_0\longmapsto h_0+2t\,\bary(\tau_{h_0}\mu)
    $$
    is onto. Its Jacobian is
    $$
        J_F(h_0)=I_n+2t\,\Cov(\tau_{h_0}\mu)\succeq I_n,
    $$
    which implies that $F$ is open (it sends open sets to open sets) and proper (pre-images of compact sets are compact). Therefore $F$ is onto.
\end{proof}

\subsubsection{A perturbation result}

In view of Lemma~\ref{lem:get_rid_of_tilt}, we seek to control the covariance matrix of measures of the form
$$
\frac{d\nu_t}{d\nu}=\frac{1}{Z_t}\,e^{-t\abs{x}^2}.
$$
In our setting, $\nu$ will be a centered tilt-stable log-concave measure (a generic tilt of our original measure $\nu = \tau_{h_0}\mu$). We first record a general perturbation result valid for any probability measure~$\nu$.

\begin{lem}
    Let $\nu$ be a probability measure. Then
    $$
        \int\abs{x}^2\,d\nu_t \le \int\abs{x}^2\,d\nu.
    $$
    In particular, if $\nu$ is centered,
    \begin{equation}\label{eq:trace_decrease}
        \norm{\Cov(\nu_t)}_{op}
        \le \int\abs{x}^2\,d\nu_t
        \le \Tr(\Cov(\nu)).
    \end{equation}
\end{lem}

\begin{proof}
    We compute
    $$
        \frac{d}{dt}\int\abs{x}^2\,d\nu_t
        = -\int\abs{x}^4\,d\nu_t
          +\Big(\int\abs{x}^2\,d\nu_t\Big)^2
        \le 0.
    $$
    Moreover if $\nu$ is centered,
    $$
        \norm{\Cov(\nu_t)}_{op}
        \le \Tr(\Cov(\nu_t))
        \le \int\abs{x}^2\,d\nu_t.
    $$
\end{proof}

Our next goal is to improve \eqref{eq:trace_decrease} under the additional assumption that $\nu$ is subgaussian. The following lemma is inspired by Barthe and Milman~\cite{barthe2013transference}.

\begin{lem}\label{lem:general_perturbation_subg}
    Let $\nu$ be a centered subgaussian probability measure on $\R^n$ with subgaussian constant $\sigma_{SG}(\nu)$. Let $\bar\nu$ be another probability measure absolutely continuous with respect to $\nu$, and set
    $$
        R=\frac{d\bar\nu}{d\nu},
        \qquad
        K=\E_\nu R^2.
    $$
    Then
    $$
        \sigma_{SG}^2(\bar\nu)
        \ \lesssim\
        \sigma_{SG}^2(\nu)\left(1+\log K\right).
    $$
\end{lem}

\begin{proof}
    For $\lambda>0$ let $E_\lambda=\{R\le \lambda\}$. For any measurable $S\subset\R^n$,
    \begin{align*}
        \bar\nu(S)
        &=\int_{S\cap E_\lambda}R\,d\nu+\int_{S\cap E_\lambda^c}R\,d\nu\\
        &\le \lambda\,\nu(S) + \bar\nu(E_\lambda^c).
    \end{align*}
    By Markov's inequality,
    $$
        \bar\nu(E_\lambda^c)
        =\bar\nu(R>\lambda)
        \le \frac{\E_{\bar\nu}(R)}{\lambda}
        = \frac{K}{\lambda}.
    $$
    Fix $\theta\in\mathcal{S}^{n-1}$ and $r>0$, and let $S_{\theta,r}=\{x:\abs{x\cdot\theta}\le r\}$. Using the subgaussianity of $\nu$,
    \begin{align*}
        \bar\nu(S_{\theta,r})
        &\le \lambda\,\nu(S_{\theta,r}) + \frac{K}{\lambda}\\
        &\le 2\lambda\,\exp\!\left(-\frac{c r^2}{\sigma_{SG}^2(\nu)}\right)
            + \frac{K}{\lambda}.
    \end{align*}
    Optimizing over $\lambda>0$ yields
    \begin{align*}
        \bar\nu(S_{\theta,r})
        &\le \min\left(\sqrt{2K}\,\exp\!\left(-\frac{c r^2}{2\sigma_{SG}^2(\nu)}\right), \, 1\right)\\
        &\le 2\,\exp\!\left(-\frac{c_1 r^2}{\sigma_{SG}^2(\nu)(1+\log K)}\right),
    \end{align*}
    for a universal constant $c_1>0$. This gives
    $$
        \norm{\langle\cdot,\theta\rangle}_{L^{\psi_2}(\bar\nu)}^2
        \ \lesssim\
        \sigma_{SG}^2(\nu)\left(1+\log K\right).
    $$
    Since $\langle\cdot,\theta\rangle$ need not be centered under $\bar\nu$, we apply Lemma~\ref{lem:centering_subgauss} to conclude.
\end{proof}

We now specialize to the setting of interest.

\begin{lem}\label{lem:barthe_milman_subg}
    Let $\nu$ be a centered subgaussian probability measure on $\R^n$ with subgaussian constant $\sigma_{SG}(\nu)$, and let $t>0$. Define
    $$
        \nu_t=\frac{1}{Z_t}\,e^{-t\abs{x}^2}\,\nu,
        \qquad
        K(t)=\frac{\int e^{-2t\abs{x}^2}d\nu}{\Big(\int e^{-t\abs{x}^2}d\nu\Big)^2}.
    $$
    Then
    $$
        \sigma_{SG}^2(\nu_t)
        \ \lesssim\
        \sigma_{SG}^2(\nu)\left(1+\log K(t)\right).
    $$
\end{lem}
We now turn to estimating $K(t)$. We are primarily interested in the regime of small $t$, since for large $t$ and for log-concave $\nu$ (the case relevant to us), we may simply use the bound $\sigma_{SG}^2(\nu_t)\le 1/t$. When $\nu$ is the standard Gaussian, a direct computation shows that $K(t)$ satisfies 
$$\log K(t)\ \lesssim\ n t^2 \;=\; \E(\abs{G}^2)\,t^2.$$
We will recover this estimate up to an additional factor involving the $\psi_2$ norm of $\abs{X}-\E\abs{X}$; see Lemma~\ref{lem:K_estimate} below. We begin with a preliminary lemma.
\begin{lem}\label{lem:deviation_norm}
Let $X$ be a random vector with subgaussian norm then for any $r>0$,

$$\PP\left(\abs{X}^2 \leq \E\abs{X}^2 - r\right) \leq \ 2\exp\left(-\frac{cr^2}{\E\abs{X}^2L^2}\right) $$
where $c>0$ is a universal constant and
$$L =\norm{ \ \abs{X} - \E\abs{X}\ }_{\psi_2}.$$
\end{lem}

\begin{proof}
    For any $0<r<\E\abs{X}^2$ we have,
    \begin{align*}
        \PP\left(\abs{X}^2 \leq \E\abs{X}^2 - r\right) & = \PP\left(\sqrt{X} \leq \sqrt{\E\abs{X}^2 - r}\right)\\
        & \leq \PP\left(\abs{X} \leq \left(\E\abs{X}^2\right)^{1/2} - \frac{r}{2\left(\E\abs{X}^2\right)^{1/2}}\right)\\
        &\leq \PP\left(\abs{X} \leq \E\abs{X} + L - \frac{r}{2\left(\E\abs{X}^2\right)^{1/2}}\right)
    \end{align*} 
    where in the first inequality we used the concavity of the square-root function, and in the second one, we used that 
    \begin{align*}
        \E\abs{X}^2 &= (\E\abs{X})^2 + \Var(\abs{X}) \leq (\E\abs{X})^2 + L^2 \leq (\E\abs{X}+L)^2. 
    \end{align*}
    Using the Gaussian concentration for $\abs{X}$, we get that for any $r\geq 4L\left(\E\abs{X}^2\right)^{1/2}$,
    \begin{align}
        \PP\left(\abs{X}^2 \leq \E\abs{X}^2 - r\right)  & \leq \PP\left(\abs{X} \leq \E\abs{X} - \frac{r}{4\left(\E\abs{X}^2\right)^{1/2}}\right)\\
        &\leq \exp\left(-\frac{r^2}{16\E\abs{X}^2L^2}\right)
    \end{align}
    Combining this with the trivial bound $\PP\left(\abs{X}^2 \leq \E\abs{X}^2 - r\right)\leq 1$ for small $r$, yields the result with e.g $c=\frac{1}{32}$.
\end{proof}
\begin{lem}\label{lem:K_estimate}
Under the same hypothesis as in Lemma \ref{lem:barthe_milman_subg}, we have
\[
    \log K(t) \lesssim
    1 + t^2\, \norm{ \abs{X} - \E\abs{X} }_{\psi_2}^2 \, \E\abs{X}^2.
\]
where $X\sim\nu$.
\end{lem}

\begin{proof}
    By Jensen's inequality, $\int_{\R^n} e^{-t\abs{x}^2}\,d\nu = \E e^{-t|X|^2} \geq e^{-t\E\abs{X}^2}$, so that
    \begin{equation}\label{eq:K(t)}
        K(t) \leq e^{2t\E\abs{X}^2}\,\E e^{-2t|X|^2}.
    \end{equation}
    Now, using Lemma \ref{lem:deviation_norm}, for any $t>0$,
\begin{align*}
\E\left[e^{-t\abs{X}^2}\right]
&= \int_0^\infty \PP\left(\abs{X}^2\le r\right)\, t e^{-tr}\,dr\\
&\le t\int_{0}^{\E\abs{X}^2}\PP(\abs{X}^2\le r)\, e^{-tr}\,dr + e^{-t\E\abs{X}^2}\\
&= t\int_{0}^{\E\abs{X}^2}\PP(\abs{X}^2\le \E\abs{X}^2 - r)\, e^{-t(\E\abs{X}^2 - r)}\,dr + e^{-t\E\abs{X}^2}\\
&\le 2t e^{-t\E\abs{X}^2}\int_{0}^{\E\abs{X}^2}
    \exp\!\left(-\frac{c r^2}{\E\abs{X}^2 L^2} + tr\right)dr + e^{-t\E\abs{X}^2}.
\end{align*}
Completing the square,
\[
-\frac{c r^2}{\E\abs{X}^2 L^2}+tr
\;=\;
-\frac{c}{\E\abs{X}^2 L^2}\Big(r-\frac{t\,\E\abs{X}^2 L^2}{2c}\Big)^2
\;+\;\frac{t^2\,\E\abs{X}^2L^2}{4c}.
\]
Hence,
\begin{align*}
\E\left[e^{-t\abs{X}^2}\right]
&\le
2t\,e^{-t\E\abs{X}^2}
\exp\!\left(\frac{t^2\,\E\abs{X}^2L^2}{4c}\right)
\int_{\R}
\exp\!\left(
-\frac{c}{\E\abs{X}^2 L^2}
\Big(r-\tfrac{t\,\E\abs{X}^2L^2}{2c}\Big)^2
\right)
dr
\;+\;
e^{-t\E\abs{X}^2}\\
&\le
2e^{-t\E\abs{X}^2}\left(
2t\sqrt{\frac{\pi\,\E\abs{X}^2L^2}{c}}\,
e^{\frac{t^2\,\E\abs{X}^2L^2}{4c}}
+1
\right).
\end{align*}
Using that $z\le e^{z^2}$ for all $z\in\R$, we obtain
\[
\E\left[e^{-t\abs{X}^2}\right]
\;\le\;
8\sqrt{\pi}\, e^{-t\E\abs{X}^2}
\left(1 + e^{\frac{t^2\,\E\abs{X}^2L^2}{2c}}\right).
\]
Combining this with \eqref{eq:K(t)} yields
\[
K(t) \leq 8\sqrt{\pi}\left(1 + e^{\frac{2t^2\,\E\abs{X}^2L^2}{c}}\right),
\]
which concludes the proof.
\end{proof}
Putting everything together, we arrive at the following perturbation estimate, which is essentially sharp.

\begin{prop}\label{prop_760}
    Let $\nu$ be a centered log-concave probability measure on $\R^n$, and for $t\ge0$ let
    \[
        \nu_t=\frac{1}{Z_t}\,e^{-t\abs{x}^2}\,\nu
    \]
    denote its Gaussian perturbation.
    Let $X\sim\nu$, and assume for normalization that $\E |X|^2 = n$.
    Then
    \[
        \sigma_{SG}^2(\nu_t)
        \ \lesssim\ 
        \min\!\left(
            \sigma_{SG}^2(\nu)\bigl(1+t^2 n \norm{\abs{X}-\E\abs{X}}_{L^{\psi_2}(\nu)}^2\bigr)
            \ ,\ 
            \frac{1}{t}
        \right).
    \]
Moreover, this estimate is sharp: there exists a sequence of log-concave probability measures for which
even the smaller quantity $\norm{\Cov(\nu_t)}_{op}$ matches the upper bound, up to a universal constant.
\end{prop}
\begin{proof}
    We always have 
    $$\sigma_{SG}^2(\nu) \leq \frac{1}{t}$$
    by Bakry-Emery and Herbst's argument. The other term in the minimum comes from Lemma \ref{lem:barthe_milman_subg} together with Lemma \ref{lem:K_estimate}. 

  For the sharpness, consider the convex body $K\subset\R^{n+1}$ defined by
    \[
        K=\sqrt{3}\,B_\infty^{n}\times\R \ \cap\ 
        \{(x,\lambda):\ \abs{x}\le \sqrt{n}+C_0(1-\abs{\lambda})\},
    \]
    where $C_0>0$ is a universal constant.  
    Let $X$ be uniformly distributed on $K$, and denote its law by $\nu$.  
    It is not too hard to check that
    \[
        \E\abs{X}^2 \lesssim n,\qquad
        \sigma_{SG}^2(X)\lesssim 1,\qquad
        \norm{\abs{X}-\E\abs{X}}_{\psi_2}^2 \lesssim 1.
    \]

    As before, write $\nu_t$ for the Gaussian perturbation of $\nu$, and let $X_t\sim\nu_t$.  
    For a large enough choice of $C_0>0$, a careful inspection of the geometry of $K$ along the $e_{n+1}$ coordinate, yields
    \[
        \norm{X_t\cdot e_{n+1}}_{\psi_2}^2 
        \ \ge\ \Var(X_t\cdot e_{n+1})
        \ \gtrsim\ \min\!\left(1+t^2 n\ ,\ \frac{1}{t}\right).
    \]
    This matches the upper-bound. We omit the details of this straightforward, but somewhat tedious analysis.
\end{proof}
We now complete the proof of Lemma~\ref{lem:intro_tilt_implies_strong_tilt}, from which
Theorem~\ref{thm:tilt_logsob} will follow together with Lemma~\ref{lem:intro_strong_tilt_implies_logsob}.

\subsection{Proof of Lemma \ref{lem:intro_tilt_implies_strong_tilt}}\label{last_subsec}

Let $\mu$ be a $1$ tilt-stable log-concave probability measure.  
Fix $t>0$ and $h\in\R^n$.  
By Lemma~\ref{lem:get_rid_of_tilt}, there exists $h_0\in\R^n$ such that
\[
    \mu_{t,h}
    =\frac{1}{Z}\,(\tau_{h_0}\mu)\, e^{-t\abs{x-\bary(\tau_{h_0}\mu)}^2}.
\]
Write $A_{t,h}=\Cov(\mu_{t,h})$.  
Since $\tau_{h_0}\mu$ is again $1$ tilt-stable, it is $2$-subgaussian (Lemma \ref{lem315}). Applying Proposition~\ref{prop_760} yields
\begin{align*}
    \norm{A_{t,h}}_{op} 
    &\leq\sigma_{SG}^2(\mu_{t,h})\\
    &\lesssim \min\!\left(1+t^2\Tilde{K_n}^2\Tr(A_{t,h})\ ,\ \frac{1}{t}\right)\\
    &\lesssim \min\!\left(1+t^2\Tilde{K_n}^2 n\ ,\ \frac{1}{t}\right)\\
    &\lesssim n^{1/3}\Tilde{K_n}^{2/3}.
\end{align*}
Since $t$ and $h$ were arbitrary, we conclude that $\mu$ is 
$C\,n^{1/6}\Tilde{K_n}^{1/3}$ strongly tilt-stable for some universal constant $C>0$.
\bibliographystyle{plain}
\bibliography{biblio}
\end{document}